\documentclass[12pt, oneside, psamsfonts]{article}
\usepackage[T1]{fontenc}
\usepackage{lmodern}
\usepackage{amsmath,amssymb,amsthm,mathtools}
\usepackage{microtype,enumitem}
\usepackage[colorlinks=true,linkcolor=blue,citecolor=blue,urlcolor=blue]{hyperref}

\usepackage[scale=0.8]{geometry}

\newtheorem{theorem}{Theorem}[section]
\newtheorem{proposition}[theorem]{Proposition}
\newtheorem{lemma}[theorem]{Lemma}
\newtheorem{corollary}[theorem]{Corollary}
\theoremstyle{definition}
\newtheorem{definition}[theorem]{Definition}
\theoremstyle{remark}
\newtheorem{remark}[theorem]{Remark}

\newcommand{\mdim}{\operatorname{mdim}}
\newcommand{\rc}{\operatorname{rc}}
\newcommand{\supp}{\operatorname{supp}}
\newcommand{\rank}{\operatorname{rank}}
\newcommand{\car}{\operatorname{car}}
\newcommand{\Tr}{\operatorname{Tr}}
\newcommand{\Mu}{\mathcal M_\Gamma(X)}
\newcommand{\UU}{\mathcal U}
\newcommand{\DD}{\mathcal D}
\newcommand{\C}{\mathbb C}

\newcommand{\N}{\mathbb N}
\newcommand{\Z}{\mathbb Z}
\newcommand{\eps}{\varepsilon}
\newcommand{\bpre}{\precsim}

\newcommand{\kk}{\kappa}

\title{Mean dimension and radius of comparison}

\author{George A. Elliott and Zhuang Niu}

\date{October 7, 2026}

\begin{document}
\maketitle

\begin{abstract}
Consider a minimal and free topological dynamical system $(X, \Gamma)$, where $\Gamma$ is a discrete amenable group. Assume $(X, \Gamma)$ has the uniform Rokhlin property (URP) and the Cuntz comparison of open sets (COS), then it is shown that
$$\mathrm{rc}(\mathrm{C}(X) \rtimes \Gamma) = \frac{1}{2} \mathrm{mdim}(X, \Gamma).$$
The same statement also holds for unital simple AH algebras with diagonal maps.

The arguments in both cases are based on the recent preprint of OpenAI dealing with $\mathbb Z$-actions and the topological ideas used in that paper.
\end{abstract}

\tableofcontents

\section{Introduction and main results}
Mean dimension measures the complexity of propagated finite open covers.
For a diagonal AH presentation, propagation is along its eigenvalue maps; for
an amenable action, propagation is along partial orbits. It was proved
that the radius of comparison is at most half of the corresponding mean
dimension for simple diagonal AH algebras (\cite[Theorem 6.2]{NiuAH})
and for actions with the uniform Rokhlin property (URP) and Cuntz comparison
of open sets (COS) (\cite[Theorem 4.8]{Niu}). OpenAI's recent preprint
\cite{OAI-main} proves the reverse bound for minimal $\mathbb Z$-actions,
using the torus compression theorem of its companion preprint
\cite[Theorem 1.3]{OAI-top}.

We extend this reverse-bound to diagonal AH algebras and
amenable actions with URP. In the first setting, eigenvalue-map occurrences
provide the coordinates of the comparison construction. In the second case,
exact finite-shape tilings supply an allocation suited to arbitrary
amenable groups:

\begin{theorem}[Diagonal AH equality]\label{thm:ah-main}
Let $\mathcal P=(A_i,\phi_i)$ be a unital inductive system of finite
direct sums of homogeneous corners over connected compact metrizable
spaces, with ordinary globally diagonal maps as in
Definition~\ref{def:ah}. Suppose that $A=\varinjlim\mathcal P$ is simple.
For the mean dimension $\mathrm{mdim}(A_i, \phi_i)$ of
\cite[Definition 3.6]{NiuAH}, recalled in Definition~\ref{def:gamma},
\[
 \operatorname{rc}(A)=\frac12\mathrm{mdim}(A_i, \phi_i)
 \qquad\text{in }[0,\infty].
\]
\end{theorem}

\begin{theorem}[URP lower bound and URP/COS equality]
\label{thm:actions-main}
Let $(X, \Gamma)$ be a free and minimal dynamical system, where $\Gamma$ is a discrete amenable group, and $X$ is a compact metrizable space. If $(X, \Gamma)$ has URP, then
\[
 \operatorname{rc}(C(X)\rtimes\Gamma) \geq \frac{1}{2}\operatorname{mdim}(X,\Gamma).
\]
If $(X, \Gamma)$ also has COS, then
\[
 \operatorname{rc}(C(X)\rtimes\Gamma)
 =\frac12\operatorname{mdim}(X,\Gamma)
 \qquad\text{in }[0,\infty].
\]
\end{theorem}
Among other actions, all minimal and free $\mathbb Z^d$-actions have URP and COS (\cite{Niu-Zd}).

The new assertion in each theorem is the reverse bound. For every
finite comparison bound, we control the limiting rate of each finite
open cover. Taking the supremum over covers therefore also treats
infinite mean dimension. The AH equality identifies the mean dimension
of every presentation in the stated class with the same algebraic
invariant. In the dynamical equality, COS is used only for the 
upper bound.

Closely related results in this direction were obtained earlier by Hirshberg and Phillips in \cite{HP}. 

The following corollaries point out that a certain property of a given system is equivalent to a property of the associated C*-algebra. 


\begin{corollary}
Let $(A_i,\phi_i)$ be an AH system with diagonal maps. Let $A$ be the limit algebra, and let $D$ be its canonical diagonal subalgebra. Then the following properties are equivalent:
\begin{enumerate}
\item $A$ is $\mathcal Z$-absorbing,
\item $\mathrm{mdim}(A_i, \phi_i) = 0$.
\end{enumerate}
\end{corollary}

\begin{corollary}
Let $(X, \Gamma)$ be a free and minimal dynamical system, where $\Gamma$ is a discrete amenable group, and $X$ is a compact metrizable space. Then the following properties are equivalent:
\begin{enumerate}
\item $\mathrm{C}(X) \rtimes \Gamma$ is $\mathcal Z$-absorbing,
\item $\mathrm{mdim}(X, \Gamma) = 0$,
\item $(X, \Gamma)$ has the small boundary property
\item the pair $(\mathrm{C}(X), \mathrm{C}(X) \rtimes \Gamma)$ has property (C) of \cite{EN-C}.
\end{enumerate}
\end{corollary}

\begin{remark}[Compression theorem and scope]
The main technical ingredient is \cite[Theorem 1.3]{OAI-top}, restated as
Theorem~\ref{input:compression}. Section~\ref{sec:torus-compression}
defines its bundle $E_f$ over the entire sparse-parameter locus.
Lemma~\ref{lem:ah-embedding} verifies the embedding hypothesis for AH
blocks, and Section~\ref{sec:compression-application} identifies all
the data in the dynamical application. The comparison-to-bundle
construction and colored packet reconstruction are adapted from
\cite[Sections 3--5]{OAI-main}.

\end{remark}

The common tools for AH algebras and crossed product C*-algebras are discussed in Sections~\ref{sec:common-conventions}--\ref{sec:common-packets}. In
Sections~\ref{sec:ah-definitions}--\ref{sec:ah-consequences}, the theorem is proved for AH algebras.  The crossed product C*-algebras are handled in Sections~\ref{sec:actions-conventions}--\ref{sec:individual-cover}. The AH argument decreases the maximum cover rate over all later AH components, while argument on the crossed product C*-algebras decreases the rate of an entire orbit cover, using exact group tilings, padded slots, and weighted sparse-coordinate compression.

\subsubsection*{Acknowledgements} 

The work was carried out using ChatGPT. The first author is supported by an NSERC Discovery Grant, and the second author is supported by a Simons Foundation Grant (MP-TSM-00002606)

\section{Comparison and cover conventions}\label{sec:common-conventions}
Throughout the paper, normalized traces on a unital algebra are extended to matrix
amplifications using the unnormalized matrix trace. For each 
$b\in M_\infty(A)_+$ and $\tau\in T(A)$, put
\[
 d_\tau(b)=\lim_{m\to\infty}(\tau\otimes\Tr)(b^{1/m}).
\]
A finite $r\geq0$ is a comparison bound if
\[
 d_\tau(b_0)+r<d_\tau(b_1)\quad(\tau\in T(A))
 \quad\Longrightarrow\quad b_0\bpre b_1
\]
for all positive $b_0,b_1\in M_\infty(A)$. Here $\bpre$ denotes
Cuntz subequivalence. The infimum of the comparison bounds, with the
infimum of the empty set equal to infinity, is $\rc(A)$. ( Since all C*-algebras in this paper are nuclear, we use trace in the definition of rc.)
For a rank-$R$ projection in an amplification of a homogeneous corner
whose unit has rank $n$, the normalized fiber trace has value $R/n$.

For a finite open cover $\alpha$ of a nonempty compact metrizable space, let
$\DD(\alpha)$ be the minimum order of a finite open refinement, where
order is the maximum number of members through a point minus one.
We use the standard inequalities
\[
 \DD(\alpha\vee\beta)\leq\DD(\alpha)+\DD(\beta),
 \qquad \DD(f^{-1}\alpha)\leq\DD(\alpha).
\]
Each application specifies the propagation of its covers and the normalization
of their limiting rates.

\begin{lemma}[A fixed compression for weighted coordinate supports]
\label{lem:weighted-coordinate-compression}
Let $\mathcal I$ be a finite set, let $d_i\geq 0$ be integers, and let
$q\geq 0$ be an integer.  There is a complex linear map
\[
 C:\bigoplus_{i\in\mathcal I}\mathbb C^{d_i}\longrightarrow
 \mathbb C^q
\]
whose restriction to each coordinate subspace
$\bigoplus_{i\in\Lambda}\mathbb C^{d_i}$ with
$\sum_{i\in\Lambda}d_i\leq q$ is injective.
\end{lemma}

\begin{proof}
For each such $\Lambda$ with positive total dimension, failure of
injectivity is a proper determinantal algebraic subset of the space of
all linear maps to $\mathbb C^q$: it is proper because that particular
coordinate subspace has dimension at most $q$.  There are only finitely
many subsets $\Lambda$.  A finite union of proper algebraic subsets
cannot exhaust this complex vector space.  Any map outside the union
works.  Zero-dimensional coordinate subspaces impose no condition; in
particular the assertion for $q=0$ is immediate.
\end{proof}

\section{Boundary-preserving torus compression}\label{sec:torus-compression}
For $p\geq1$ set $t=2p+1$ and $C_t=[-1,1]^t$. Choose the geometric data of \cite[Notation 1.2 and Lemma 5.1]{OAI-top}: a fixed smooth hypersurface
\[
 L\cong(\mathbb T^2)^p\subset(-1/32,1/32)^t
\]
with a framed real normal line, and the rank-$p$ complex vector bundle
\[
 \mathcal H=\bigoplus_{i=1}^p\lambda_i\subset L\times\C^{2p},
 \qquad
 \left\langle\prod_{i=1}^pc_1(\lambda_i),[L]\right\rangle=1.
\]
Here each $\lambda_i\subset L\times\C^2$ is the pullback of a degree-one line bundle from the $i$th two-torus factor; $L$ has the product orientation. Denote the orthogonal projection field of $\mathcal H$ by $Q:L\to M_{2p}(\C)$. On $L^n$ define
\[
 \mathcal H_n=\bigoplus_{\nu=1}^n\operatorname{pr}_\nu^*\mathcal H
 \subset L^n\times\C^{2pn},
\]
where $\operatorname{pr}_\nu$ selects the $\nu$th torus slot. In the notation of \cite{OAI-top}, $C(t,n)=C_t^n$ and $T(t,n)=([-2,2]^t)^n$.

For $s\geq0$ define
\[
 Z_s(t,n)=\{z\in([-2,2]^t)^n:|\{j:z_j\ne0\}|\leq\lfloor s\rfloor\}.
\]
For a continuous $f:Y\to C_t^n$, let
\[
 \mathcal T(f,Z_s)=\{(y,z):z\in Z_s(t,n),\ f(y)-z\in L^n\},
 \qquad d_f(y,z)=f(y)-z.
\]
The bundle denoted by $E_f$ in \cite[Notation 1.2]{OAI-top} is the pullback
\[
 \mathcal E_f=d_f^*\mathcal H_n
 \longrightarrow\mathcal T(f,Z_s),
 \qquad
 (\mathcal E_f)_{(y,z)}
 =\bigoplus_{\nu=1}^n\mathcal H_{f_\nu(y)-z_\nu}.
\]
Thus $\mathcal E_f$ has complex rank $pn$, with a canonical inclusion in the trivial bundle of rank $2pn$. Its projection in that ambient bundle is
\[
 \Pi_f(y,z)=\operatorname{diag}_{\nu=1}^n
       Q(f_\nu(y)-z_\nu).
\]
In the following theorem we take $s=an$. The required embedding has a target of rank strictly smaller than the canonical ambient rank $2pn$, and must be continuous simultaneously in both $y$ and $z$ across the entire sparse-parameter locus.

The following is Theorem 1.3 of \cite{OAI-top}.
\begin{theorem}[Torus compression]\label{input:compression}
Let $p,n\geq1$ be integers, put $t=2p+1$, and choose $L,\mathcal H$ as above. Let $Y$ be compact metrizable, $0<a<1/16$, and $f:Y\to C_t^n$ continuous. Suppose there is one continuous complex bundle monomorphism
\[
 \mathcal E_f\longrightarrow
 \mathcal T(f,Z_{an})\times\C^q,
 \qquad q\in\Z_{\geq0},\quad q<2pn,
\]
on the entire compact sparse-parameter locus. Then there is an integer $u\geq1$ and a continuous map $G:Y^u\to C_t^{un}$ such that at every point at least $(a/8)un$ slots are on $\partial C_t$, and each original boundary slot of $f^{\times u}$ is preserved as an entire vector. The cited theorem allows $u$ arbitrarily large; only existence is needed here.
\end{theorem}

The theorem applies to arbitrary compact metrizable parameter spaces.
In both applications below, the rank saving is supplied by a single
bundle monomorphism continuous in the base point and in the sparse
parameter. Exact preservation of a whole boundary slot is the
relative condition needed for packet reconstruction.

\section{Colored covers and reconstruction}\label{sec:common-packets}
The following elementary lemmas reproduce the cover mechanism. They apply both to diagonal eigenmap factors and to finite orbit covers.

\begin{lemma}[Polyhedral refinement]\label{lem:refinement}
Let $X$ be nonempty compact metrizable and let $\phi:X\to P$ be continuous, where $P$ is a finite polyhedron of dimension at most $q$. If each point of $\phi(X)$ has an open neighborhood whose whole inverse image lies in a member of a finite open cover $\mathcal V$, then $\DD(\mathcal V)\leq q$.
\end{lemma}
\begin{proof}
Choose finitely many such neighborhoods, add $P\setminus\phi(X)$, and subdivide $P$ so that its vertex stars refine this cover. The pulled-back vertex stars form a refinement of $\mathcal V$ of order at most $q$.
\end{proof}

\begin{lemma}[Colored nerve]\label{lem:nerve}
For a finite open cover $\mathcal V$ of a nonempty compact metrizable space $X$ and $m=\DD(\mathcal V)$ there is a continuous map $\varphi:X\to S$ to a finite simplicial complex of dimension at most $m$. Its vertices have colors $1,\ldots,m+1$, distinct on each simplex. Each vertex $v$ is assigned a member $V(v)\in\mathcal V$ such that
\[
 \lambda_v(\varphi(x))>0\quad\Longrightarrow\quad x\in V(v).
\]
\end{lemma}
\begin{proof}
Take the nerve of a refinement of order $m$, with a subordinate partition-of-unity map, and barycentrically subdivide. Color the vertex corresponding to a face $F$ by $\dim F+1$. Assign it any one vertex of $F$ and the original cover member assigned to that vertex. Positivity of the new coordinate implies positivity of the chosen original coordinate.
\end{proof}

\begin{lemma}[Exact packet reconstruction]\label{lem:packets}
Let $S$ be as above, fix $t\geq1$, and put $g=\lfloor(m+1)/(t+1)\rfloor$. There are continuous maps $\mathfrak f_j:S\to C_t$, $1\leq j\leq g$, with this property: given a continuous $y:Y\to S$ and continuous $v_j:Y\to C_t$ satisfying
\[
 v_j(z)=\mathfrak f_j(y(z))\quad\text{when }\mathfrak f_j(y(z))\in\partial C_t,
\]
there is a continuous $\widetilde y:Y\to S$ with vertex support contained in the original vertex support and satisfying
\[
 \dim\car(\widetilde y(z))\leq
 m-|\{j:v_j(z)\in\partial C_t\}|.
\]
\end{lemma}
\begin{proof}
Fix a boundary-preserving homeomorphism $\kk:\Delta^t\to C_t$. Let $c_\ell$ be the sum of vertex coordinates of color $\ell$. Split the first $(t+1)g$ colors into packets, with packet weight vector $w_j$ and total mass $M_j$. Set $b=(1/(t+1),\ldots,1/(t+1))$, and choose $0<\tau<1/(4(m+1))$. The continuous simplex vectors
\[
 q_j=\begin{cases}w_j/M_j,&M_j\geq\tau,\\
 b+(w_j-M_jb)/\tau,&M_j\leq\tau
 \end{cases}
\]
give $\mathfrak f_j=\kk\circ q_j$.

Choose $\eta:[0,1]\to[0,1]$ continuous, zero on $[0,\tau]$ and one on $[2\tau,1]$. Replace packet $j$ by the unnormalized vector $\eta(M_j)M_j\kk^{-1}(v_j)$ and leave leftover colors unchanged. Let $\widehat c_\ell$ be the resulting color weights, with sum $Z$. Then $Z\geq1-2g\tau>0$. If an original color is absent, its new weight is zero: when the packet survives, $M_j>\tau$ and its original normalized vector is on the simplex boundary, so the boundary-preservation assumption fixes that vector.

For a vertex $v$ of color $\ell$, define its new coordinate to be
\[
 \widetilde\lambda_v=\begin{cases}
 \lambda_v\widehat c_\ell/(c_\ell Z),&c_\ell>0,\\
 0,&c_\ell=0.
 \end{cases}
\]
At $c_\ell=0$ this is continuous: $\lambda_v/c_\ell\leq1$, whereas $\widehat c_\ell\to0$ and $Z$ is bounded below. The new coordinates sum to one and use only original vertices, so they define a point of $S$. Each boundary output removes at least one color from its packet, or the packet is discarded entirely. The packets are disjoint and each color occurs at most once in a carrier, giving the asserted bound.
\end{proof}

\begin{remark}
The saving is measured from the common bound $m$. Thus a boundary input
already present is counted against the full palette of $m+1$ colors,
as required in both applications.
\end{remark}

\part{Diagonal AH algebras}\label{part:ah}

\section{Diagonal presentations and their mean dimension}\label{sec:ah-definitions}

\begin{definition}[Ordinary diagonal corner presentation]\label{def:ah}
An ordinary diagonal corner AH presentation is a unital inductive
system $\mathcal P=(A_i,\phi_{i,i+1})$ with
\[
 A_i=\bigoplus_{s=1}^{h_i}
       p_{i,s}M_{r_{i,s}}(C(X_{i,s}))p_{i,s},
 \qquad A=\varinjlim_i A_i.
\]
Here $h_i,r_{i,s}$ are positive integers, $X_{i,s}$ is a nonempty
compact connected metrizable space, and $p_{i,s}$ is a continuous
projection of nonzero constant rank $n_{i,s}$.  Write $E_{i,s}$ for
its range bundle.  Injectivity of the connecting maps is not assumed.

For every adjacent pair of stages and every target summand there is
a finite occurrence list of continuous eigenmaps, each with a specified
source summand, and a global unitary bundle identification
\[
 E_{i+1,k}\cong
 \bigoplus_{s=1}^{h_i}\ \bigoplus_{a=1}^{m_{i,i+1}^{s,k}}
       (\lambda_{i,i+1}^{s,k,a})^*E_{i,s},
 \qquad
 \lambda_{i,i+1}^{s,k,a}:X_{i+1,k}\longrightarrow X_{i,s}.
\]
Under this identification, the connecting map is the direct sum of
the operators $f_s(\lambda_{i,i+1}^{s,k,a}(x))$.  In particular,
there are no extra multiplicity bundles or twisting projections.
This is the ordinary diagonal class of \cite{NiuAH}, as distinguished
from the generalized diagonal maps in its Section~7.

We fix these adjacent occurrence lists and obtain the lists for
$\phi_{i,j}$ by composition and concatenation along all paths.
Occurrences are retained separately even when their eigenmaps agree.
Thus for $j\geq i$ the target bundle $E_{j,k}$ is globally identified
with the direct sum of
$(\lambda_{i,j}^{s,k,a})^*E_{i,s}$, where
$1\leq a\leq m_{i,j}^{s,k}$, and
\begin{equation}\label{eq:ah-rank-identity}
 n_{j,k}=\sum_{s=1}^{h_i}m_{i,j}^{s,k}n_{i,s}.
\end{equation}
The identity map at a stage has one identity occurrence for each
summand.  The bundle identifications for later maps are those induced
by the adjacent identifications.
\end{definition}

Let $X_i=\bigsqcup_{s=1}^{h_i}X_{i,s}$.  Covers of $X_i$ whose
members lie in individual summands form a cofinal family of finite
open covers.  Such a cover is equivalently a tuple
$\alpha=(\alpha_s)_{s=1}^{h_i}$ of finite open covers of $X_{i,s}$.
Define its propagated cover on $X_{j,k}$ by
\begin{equation}\label{eq:ah-propagated-cover}
 \alpha_{i,j,k}
  =\bigvee_{s=1}^{h_i}\ \bigvee_{a=1}^{m_{i,j}^{s,k}}
        (\lambda_{i,j}^{s,k,a})^{-1}\alpha_s.
\end{equation}
An empty join contributes the one-member cover
$\{X_{j,k}\}$.  In particular $\alpha_{i,i,s}=\alpha_s$.

\begin{lemma}[Monotonicity of a propagated cover rate]
\label{lem:ah-cover-monotonicity}
For a fixed stage $i$ and cover $\alpha$, the finite numbers
\[
 b_j(\alpha)=\max_{1\leq k\leq h_j}
      \frac{\mathcal D(\alpha_{i,j,k})}{n_{j,k}},
 \qquad j\geq i,
\]
form a nonincreasing sequence.  Consequently
$d_i(\alpha)=\lim_{j\to\infty}b_j(\alpha)$ exists in $[0,\infty)$.
\end{lemma}
\begin{proof}
For $\ell\geq j\geq i$, coherence of the occurrence lists gives
\[
 \alpha_{i,\ell,t}
  =\bigvee_{k=1}^{h_j}\ \bigvee_{a=1}^{m_{j,\ell}^{k,t}}
      (\lambda_{j,\ell}^{k,t,a})^{-1}\alpha_{i,j,k}.
\]
Subadditivity and the pullback inequality yield
\[
 \mathcal D(\alpha_{i,\ell,t})
 \leq\sum_k m_{j,\ell}^{k,t}
             \mathcal D(\alpha_{i,j,k})
 \leq b_j(\alpha)\sum_km_{j,\ell}^{k,t}n_{j,k}
 =b_j(\alpha)n_{\ell,t}.
\]
Taking the maximum over $t$ proves $b_\ell(\alpha)\leq b_j(\alpha)$.
The sequence is nonnegative and bounded by $b_i(\alpha)<\infty$.
\end{proof}

\begin{definition}[Mean dimension of the diagonal presentation]
\label{def:gamma}
Following \cite[Definition~3.6]{NiuAH}, put
\[
 \gamma_i(\mathcal P)=\sup_{\alpha}d_i(\alpha)
  =\sup_{\alpha}\lim_{j\to\infty}
       \max_k\frac{\mathcal D(\alpha_{i,j,k})}{n_{j,k}},
 \qquad
 \gamma(\mathcal P)= \mathrm{mdim}(A_i, \phi_{i, i+1})=\sup_i\gamma_i(\mathcal P).
\]
The supremum is over the finite open covers just described.  The
normalizing number is the rank of the unit projection, not the size
$r_{j,k}$ of its ambient matrix algebra.  These quantities may be
infinite, although $d_i(\alpha)$ is finite for each fixed cover.

The sequence $\gamma_i(\mathcal P)$ is nondecreasing: the propagated
tuple $(\alpha_{i,i+1,k})_k$ is a cover at stage $i+1$, and its later
propagations are exactly those of $\alpha$.  Hence equivalently
$\gamma(\mathcal P)=\lim_i\gamma_i(\mathcal P)$.  We retain the
presentation in the notation; no invariance under changing the
presentation is assumed in the definition.
\end{definition}

\begin{lemma}[Uniform growth of matrix ranks]\label{lem:rank-growth}
If the unital limit $A$ is simple and infinite dimensional, then
\[
 \min_{1\leq s\leq h_i}n_{i,s}\longrightarrow\infty.
\]
This conclusion does not require injectivity of the connecting maps.
\end{lemma}
\begin{proof}
Suppose otherwise.  There are arbitrarily late stages with a summand
of rank at most some fixed integer $N$.  Passing to a subsequence of
those stages, the summand ranks have one constant value
$n\in\{1,\ldots,N\}$.  Choose a point in each such summand and an
identification of its fiber with $\mathbb C^n$.  Evaluation gives a
unital homomorphism $\rho_j:A_j\to M_n$.

For each fixed $i$, all sufficiently late such evaluations restrict
to unital homomorphisms $\rho_j\phi_{i,j}:A_i\to M_n$.
The space $H_i=\operatorname{Hom}_1(A_i,M_n)$ is compact in the
pointwise norm topology: it is a closed subset of the product of
the compact matrix balls
$\prod_{a\in A_i}\{z\in M_n:\|z\|\leq\|a\|\}$, since linearity,
multiplicativity, preservation of adjoints, and preservation of the
unit are closed conditions.  Each $H_i$ is nonempty by the late
evaluations.  In the compact product $\prod_i H_i$, the connecting
relations
\[
 \rho_i=\rho_{i+1}\phi_{i,i+1}
\]
have the finite intersection property: any finitely many relations
are satisfied by restricting one sufficiently late evaluation, with
arbitrary elements of $H_i$ chosen at the remaining stages.
There is therefore a compatible family of unital homomorphisms
$\rho_i:A_i\to M_n$.  It induces a unital homomorphism
$\rho:A\to M_n$ by the universal property of the inductive limit.
Simplicity makes $\rho$ injective, which is impossible for an
infinite-dimensional algebra mapping into $M_n$.
\end{proof}

\begin{lemma}[Finite-dimensional simple limits]\label{lem:finite-limit}
If the simple unital limit of an ordinary diagonal presentation is
finite dimensional, then
\[
 A\cong M_N,\qquad \gamma(\mathcal P)=0,
 \qquad \operatorname{rc}(A)=0.
\]
No injectivity of the connecting maps is required.
\end{lemma}
\begin{proof}
A finite-dimensional simple unital algebra is a full matrix algebra.
Fix a stage $i$ and a cover $\alpha=(\alpha_s)_s$.  The
finite-dimensional representation $\phi_{i,\infty}:A_i\to M_N$
decomposes as a finite direct sum of point-evaluation representations.
For each source summand let $S_s\subset X_{i,s}$ be its finite set
of evaluation points; this set can be empty.  If $S_s$ is nonempty,
choose pairwise disjoint open neighborhoods of its distinct points,
each contained in a member of $\alpha_s$, and let $U_s$ be their
union.  Choose $f_s:X_{i,s}\to[0,1]$ continuous, zero on $S_s$
and one on $X_{i,s}\setminus U_s$.  If $S_s$ is empty, take
$f_s=1$.  The central positive element
\[
 f=\bigoplus_s f_sp_{i,s}\in A_i
\]
has image zero in $A$.  The inductive-limit norm formula gives
$\|\phi_{i,j}(f)\|\to0$, also for noninjective systems.

For all sufficiently late $j$ this norm is less than $1/2$.
The diagonal description then implies
\[
 f_s(\lambda_{i,j}^{s,k,a}(x))<1/2
 \qquad(x\in X_{j,k})
\]
for every eigenmap occurrence.  There are consequently no occurrences
from a source summand with $S_s=\varnothing$.  Every remaining
eigenmap has its range in $U_s$.  Since $X_{j,k}$ is connected,
that range lies in one of the pairwise disjoint neighborhoods making
up $U_s$, and hence in one member of $\alpha_s$.  Its pullback
cover therefore contains the whole set $X_{j,k}$ as a member.
The joined cover $\alpha_{i,j,k}$ also contains $X_{j,k}$, so
$\mathcal D(\alpha_{i,j,k})=0$ for every target summand at all
sufficiently late stages.  Thus $d_i(\alpha)=0$ for every $i$ and
$\alpha$, proving $\gamma(\mathcal P)=0$.

Finally, in matrix algebras Cuntz comparison of positive elements is
exactly comparison of their ranks.  The unique normalized trace on
$M_N$, extended with unnormalized matrix traces to amplifications,
has dimension function $\operatorname{rank}(a)/N$.  The strict
trace inequality therefore implies Cuntz subequivalence, and
$\operatorname{rc}(M_N)=0$.
\end{proof}

\section{Comparison witnesses and the torus bundle}
\label{sec:ah-analytic}

We work with a unital diagonal AH system
\[
 A=\varinjlim(A_i,\phi_{i,j}),\qquad
 A_i=\bigoplus_s p_{i,s}M_{r_{i,s}}(C(X_{i,s}))p_{i,s},
\]
where the base spaces are compact metrizable and every unit projection
has positive constant rank.  Write $E_{i,s}=\operatorname{ran}p_{i,s}$
and $n_{i,s}=\rank p_{i,s}$.  We use normalized traces on $A$ and
unnormalized traces on external matrix amplifications.

Fix a source stage $j$.  For brevity write $X_s=X_{j,s}$,
$p_s=p_{j,s}$, $d_s=r_{j,s}$, and $n_s=n_{j,s}$.  A diagonal decomposition
of $\phi_{j,h}$ has eigenmap occurrences
\[
 \lambda_{s,k,\alpha}:X_{h,k}\longrightarrow X_s,
 \qquad 1\leq\alpha\leq m_{j,h}^{s,k}.
\]
Occurrences are indexed separately even when their eigenmaps coincide.
The diagonal decomposition gives a continuous unitary identification
\[
 E_{h,k}\cong
 \bigoplus_s\bigoplus_{\alpha=1}^{m_{j,h}^{s,k}}
                  \lambda_{s,k,\alpha}^*E_{j,s},
\]
under which the homomorphism acts by pullback on each summand.  In
particular,
\[
 n_{h,k}=\sum_s m_{j,h}^{s,k}n_s.
\]
We use the induced identifications on external amplifications as well.
This is the global diagonal structure used below; an eigenvalue pattern
alone, with additional twisted multiplicity bundles, would not supply
these identifications.

\begin{lemma}[Trivial frames in amplified corners]\label{lem:ah-frames}
Let $P\in M_d(C(X))$ be a continuous projection of constant rank
$n>0$, and let $E=\operatorname{ran}P$.  For every integer
$b\geq0$ and every $D\geq db$, the bundle $E^{\oplus D}$
contains a continuously framed trivial subbundle of rank $b$.
Equivalently, there is a continuous matrix
\[
 J:X\longrightarrow M_{Dd,b}(\C),\qquad
 J(x)^*J(x)=1_b,\quad P(x)^{\oplus D}J(x)=J(x).
\]
Consequently finitely many source blocks and finitely many prescribed
frame ranks can be accommodated in one common external amplification.
\end{lemma}
\begin{proof}
If $e_1,\ldots,e_d$ is the standard basis of $\C^d$, the section
\[
 \sigma(x)=\frac1{\sqrt n}
       (P(x)e_1,\ldots,P(x)e_d)\in E_x^{\oplus d}
\]
is continuous and has norm one, since
\[
 \sum_{a=1}^d\|P(x)e_a\|^2=\operatorname{Tr}(P(x))=n.
\]
Place $b$ copies of this section in disjoint groups of $d$ summands.
They form an orthonormal frame in $E^{\oplus db}$; adding
zero coordinates gives the stated frame in any larger amplification.
For $b=0$ the assertion is vacuous.  A maximum of the finitely many
required amplification sizes gives a common choice.
\end{proof}

\begin{lemma}[A comparison witness at one finite stage]\label{lem:ah-lift}
Let $a\in M_{D_0}(A_j)_+$ and $b\in M_{D_1}(A_j)_+$.  If their
images satisfy $a\precsim b$ in $M_\infty(A)$, then for every
$\epsilon>0$ and every prescribed index $h_0\geq j$ there are
$h\geq h_0$ and a rectangular matrix
\[
 V\in M_{D_1,D_0}(A_h)
\]
such that
\[
 \big\|V^*\phi_{j,h}(b)V-\phi_{j,h}(a)\big\|<\epsilon.
\]
The statement does not require injective connecting maps.  Once this
inequality holds, it continues to hold after pushing the same witness
to any later stage.
\end{lemma}
\begin{proof}
Denote the canonical maps into the limit by $\iota_i$.  Cuntz
subequivalence gives a rectangular $v\in M_{D_1,D_0}(A)$ with
\[
 \|v^*\iota_j(b)v-\iota_j(a)\|<\epsilon/3.
\]
One may obtain this rectangular matrix by taking the source and target
corners of a witness in a common matrix amplification, with $a$ and
$b$ padded by zeros.  Approximate $v$ by $\iota_l(v_l)$ for some
$l\geq j$ and $v_l\in M_{D_1,D_0}(A_l)$, sufficiently closely that
\[
 \|\iota_l(v_l)^*\iota_j(b)\iota_l(v_l)-\iota_j(a)\|
 <2\epsilon/3.
\]
Put $c_l=v_l^*\phi_{j,l}(b)v_l-\phi_{j,l}(a)$.  The
inductive-limit norm formula, valid also for noninjective systems, is
\[
 \|\iota_l(c_l)\|=\lim_{h\to\infty}\|\phi_{l,h}(c_l)\|.
\]
Choose $h\geq\max\{l,h_0\}$ with the last norm less than
$\epsilon$ and set $V=\phi_{l,h}(v_l)$.  Contractivity of all
subsequent connecting maps proves the final assertion.
\end{proof}

\begin{lemma}[A bundle monomorphism from diagonal comparison]
\label{lem:ah-embedding}
Let $r\geq0$ be a comparison bound for $A$.  Fix $p\geq1$,
$t=2p+1$, and $0<a<1/16$, and use the torus data
$L,\mathcal H,Q$ of Theorem~\ref{input:compression}.  At source stage
$j$, fix integers $g_s,R_s\geq0$ and continuous packet maps
\[
 F_{s,\ell}:X_s\longrightarrow C_t=[-1,1]^t,
 \qquad 1\leq\ell\leq g_s,
\]
such that
\begin{equation}\label{eq:ah-source-trace-gap}
 \frac{R_s-pg_s}{n_s}>r\qquad\text{for every source block }s.
\end{equation}
There is a witness stage $h\geq j$, which may be required to lie
beyond any prescribed index, with the following property.  For each
block $k$ of $A_h$, put
\[
 N_k=\sum_s m_{j,h}^{s,k}g_s,
 \qquad q_k=\sum_s m_{j,h}^{s,k}R_s.
\]
For $N_k\geq1$, form $f_k:X_{h,k}\to C_t^{N_k}$ by concatenating
\[
 F_{s,\ell}\circ\lambda_{s,k,\alpha},
 \qquad (s,\alpha,\ell),\quad
 1\leq\alpha\leq m_{j,h}^{s,k},\quad1\leq\ell\leq g_s,
\]
in a fixed occurrence order.  Then there is one continuous complex
bundle monomorphism
\begin{equation}\label{eq:ah-entire-locus-embedding}
 \mathcal E_{f_k}\longrightarrow
 \mathcal T(f_k,Z_{aN_k}(t,N_k))\times
             \C^{\,q_k+2p\lceil aN_k\rceil}.
\end{equation}
Here the source has rank $pN_k$, and the monomorphism is defined
simultaneously on the entire sparse torus locus.  Thus
Theorem~\ref{input:compression} applies with $Y=X_{h,k}$, $n=N_k$,
and $f=f_k$ whenever
\[
 q_k+2p\lceil aN_k\rceil<2pN_k.
\]
\end{lemma}
\begin{proof}
\emph{Positive elements with a uniform tracial gap.}
Choose a tubular neighborhood $U$ of $L$ in $(-1,1)^t$, a retraction
$\pi_L:U\to L$, and a continuous function
$\zeta:C_t\to[0,1]$ with compact support in $U$, equal to one on a
neighborhood of $L$.  Put $\widetilde Q=Q\circ\pi_L$ on $U$ and
extend $\zeta\widetilde Q$ by zero to $C_t$.  Compactness gives
$\delta>0$ such that
\begin{equation}\label{eq:ah-nearby-projection}
 y\in L,\quad \|u-y\|_\infty<\delta
 \quad\Longrightarrow\quad
 u\in U,\quad\zeta(u)=1,\quad
 \|\widetilde Q(u)-Q(y)\|<1/2.
\end{equation}
In particular, $\widetilde Q(u)$ is injective on $\mathcal H_y$.

Lemma~\ref{lem:ah-frames} gives common positive amplification sizes $D_0,D_1$
and continuous isometric frames $J_{0,s}$ and $J_{1,s}$ of ranks
$2pg_s$ and $R_s$ in $E_{j,s}^{\oplus D_0}$ and
$E_{j,s}^{\oplus D_1}$, respectively.  Define
\begin{align*}
 (P_0)_s(x)&=J_{0,s}(x)
       \operatorname{diag}_{\ell=1}^{g_s}
       \bigl(\zeta(F_{s,\ell}(x))
                          \widetilde Q(F_{s,\ell}(x))\bigr)
       J_{0,s}(x)^*,\\
 (P_1)_s(x)&=J_{1,s}(x)J_{1,s}(x)^*.
\end{align*}
The first expression uses the zero extension just specified; an empty
diagonal is zero.  Thus $P_0\in M_{D_0}(A_j)_+$ is a contraction,
and $P_1\in M_{D_1}(A_j)$ is a projection.  Their fiber ranks satisfy
\[
 \rank(P_0)_s(x)\leq pg_s,
 \qquad \rank(P_1)_s(x)=R_s.
\]
For a normalized trace $\tau$ on $A$, write its restriction to $A_j$
using component weights $t_s\geq0$, $\sum_s t_s=1$, and probability
measures $\mu_s$ on the blocks of positive weight.  The fiber trace
is normalized by $n_s$.  Functional calculus and bounded convergence
therefore give
\[
 d_\tau(P_0)\leq\sum_s t_s\frac{pg_s}{n_s},
 \qquad d_\tau(P_1)=\sum_s t_s\frac{R_s}{n_s}.
\]
Equation~\eqref{eq:ah-source-trace-gap} implies
$d_\tau(P_0)+r<d_\tau(P_1)$ for every $\tau$.  The comparison bound
gives $P_0\precsim P_1$ in the limit.  By Lemma~\ref{lem:ah-lift},
choose one stage $h$ and one rectangular witness $V$ with
\begin{equation}\label{eq:ah-finite-stage-witness}
 \|V^*\phi_{j,h}(P_1)V-\phi_{j,h}(P_0)\|<1/2.
\end{equation}
There is one such inequality in $M_{D_0}(A_h)$, valid on all its
blocks and all their points; separate witnesses are not chosen for
different sparse supports.

\emph{The finite-dimensional matrix in one target block.}
Fix $k$ with $N_k\geq1$.  Under the diagonal identification, pull back
every column of $J_{0,s}$ and $J_{1,s}$ along every occurrence
$\lambda_{s,k,\alpha}$.  Transporting these columns to the target
unit bundles gives continuous isometric frames
\[
 J_{0,k}:\C^{2pN_k}\longrightarrow
                  E_{h,k}^{\oplus D_0},\qquad
 J_{1,k}:\C^{q_k}\longrightarrow
                  E_{h,k}^{\oplus D_1}.
\]
Different occurrences give orthogonal coordinate groups, even when
the underlying maps coincide.  Writing $f_{k,\nu}$ for the slots of
$f_k$, these frames satisfy
\begin{align*}
 \phi_{j,h}(P_0)_k(x)
 &=J_{0,k}(x)\operatorname{diag}_{\nu=1}^{N_k}
       \bigl(\zeta(f_{k,\nu}(x))\widetilde Q(f_{k,\nu}(x))\bigr)
       J_{0,k}(x)^*,\\
 \phi_{j,h}(P_1)_k(x)&=J_{1,k}(x)J_{1,k}(x)^*.
\end{align*}
Let $V_k(x)$ be the block matrix of the single witness, and set
\[
 W_k(x)=J_{1,k}(x)^*V_k(x).
\]
This takes values in the fixed space $\C^{q_k}$.  Whenever
$\phi_{j,h}(P_0)_k(x)\xi=\xi$, Equation~\eqref{eq:ah-finite-stage-witness}
gives the simultaneous estimate
\begin{equation}\label{eq:ah-simultaneous-injectivity}
 \|W_k(x)\xi\|^2\geq\tfrac12\|\xi\|^2.
\end{equation}
Indeed the square of the left-hand norm is
$\langle V_k(x)^*\phi_{j,h}(P_1)_k(x)V_k(x)\xi,\xi\rangle$.
The estimate holds on the entire direct sum of the selected
eigenvalue-one ranges.

\emph{Alignment on the entire torus locus.}
Put $N=N_k$ and
\[
 \mathcal T_k=\mathcal T(f_k,Z_{aN}(t,N)),\qquad
 y_\nu=f_{k,\nu}(x)-z_\nu\in L.
\]
Choose a continuous $\chi:[-2,2]^t\to[0,1]$ with $\chi(0)=1$ and
support contained in $\{z:\|z\|_\infty<\delta\}$.  On
$\mathcal E_{f_k}=\bigoplus_\nu\mathcal H_{y_\nu}$ define
\[
 \Xi_{(x,z)}(\eta_1,\ldots,\eta_N)
 =J_{0,k}(x)
   \bigl(\chi(z_\nu)(\zeta\widetilde Q)(f_{k,\nu}(x))
                                      \eta_\nu\bigr)_{\nu=1}^N.
\]
This formula is continuous everywhere, since $\zeta\widetilde Q$
is a continuous matrix on the whole cube.  If $\chi(z_\nu)>0$,
Equation~\eqref{eq:ah-nearby-projection} applies to
$u=f_{k,\nu}(x)$ and $y=y_\nu$.  Thus $\zeta(u)=1$ and
$\widetilde Q(u)$ injects $\mathcal H_{y_\nu}$ into the corresponding
projection range.  The image of $\Xi_{(x,z)}$ consequently lies in
the eigenvalue-one subspace of $\phi_{j,h}(P_0)_k(x)$.
Distinct slots occupy distinct frame coordinates.  Therefore
\[
 \Phi_1(x,z)=W_k(x)\Xi_{(x,z)}
\]
is continuous and has the exact fiber kernel
\begin{equation}\label{eq:ah-first-map-kernel}
 \ker\Phi_1(x,z)
 =\bigoplus_{\nu:\chi(z_\nu)=0}\mathcal H_{y_\nu}.
\end{equation}
Here Equation~\eqref{eq:ah-simultaneous-injectivity} prevents
cancellation between distinct nonzero slot contributions.

\emph{One fixed map for all sparse exceptions.}
Using the fixed inclusions $\mathcal H_{y_\nu}\subset\C^{2p}$, put
\[
 \Psi_{(x,z)}(\eta_1,\ldots,\eta_N)
     =\bigl((1-\chi(z_\nu))\eta_\nu\bigr)_{\nu=1}^N
       \in(\C^{2p})^N.
\]
Every possibly nonzero output group has $z_\nu\ne0$, since
$\chi(0)=1$.  The range is consequently supported on at most
$\lfloor aN\rfloor\leq H:=\lceil aN\rceil$ groups.  There is a
single complex linear map
\[
 D:(\C^{2p})^N\longrightarrow\C^{2pH}
\]
injective on every coordinate subspace supported on at most $H$
groups.  This is Lemma~\ref{lem:weighted-coordinate-compression} with every coordinate weight equal to $2p$ and target dimension $2pH$.

Set $\Phi_2=D\Psi$.  Both $\Phi_1$ and $\Phi_2$ are continuous in
$(x,z)$ on the whole locus, including changes of sparse support.
If a vector is killed by $\Phi_1$, its components lie only in slots
with $\chi(z_\nu)=0$, by Equation~\eqref{eq:ah-first-map-kernel}.
On these components $\Psi$ is the identity.  Injectivity of $D$ on
their support therefore implies that a vector killed by both maps
is zero.  Hence $\Phi_1\oplus\Phi_2$ is the single monomorphism in
Equation~\eqref{eq:ah-entire-locus-embedding}.

Finally, $\mathcal T_k$ is compact: the base and sparse parameter set
are compact and $L^N$ is closed.  The source is the pullback of the
rank-$pN$ bundle $\mathcal H_N$.  In any local orthonormal frame its
Gram matrix under $\Phi_1\oplus\Phi_2$ is continuous and positive
definite.  Normalization by the inverse square root gives a
continuous isometric embedding, so the image is a bundle of the
stated constant rank.  If the locus is empty the conclusion is
vacuous.  Whenever the target rank is less than $2pN$, these are
exactly the hypotheses of Theorem~\ref{input:compression}.
\end{proof}

\section{Grouping later eigenmaps and reconstructing whole covers}

\begin{lemma}[All later blocks from finitely many compressed types]
\label{lem:ah-grouping}
Let $A=\varinjlim(A_q,\phi_{q,q+1})$ be a unital inductive
system of homogeneous corners
\[
 A_q=\bigoplus_{s=1}^{h_q}
 p_{q,s}M_{r_{q,s}}(C(X_{q,s}))p_{q,s},
 \qquad \rank(p_{q,s})=n_{q,s}>0,
\]
over compact metrizable spaces.  Assume that the connecting maps are
ordinary globally diagonal maps.  Denote the multiplicity from the
$s$th block of $A_q$ to the $k$th block of $A_{q'}$ by
$m_{q,q'}^{s,k}$, and choose the eigenmaps of composites by composition
of the eigenmaps of the constituent maps.

Fix an integer $t\geq1$, indices $i\leq j\leq h$, and a finite open
cover $\alpha_i$ of the
disjoint union of the $X_{i,s}$, with each cover member contained in one
component.  Write $\alpha_{i,q,s}$ for its joined pullback cover on
$X_{q,s}$ under all eigenmaps from stage $i$ to that block.  For each
block $s$ of stage $j$, put
\[
 m_s=\DD(\alpha_{i,j,s}),
 \qquad
 g_s=\left\lfloor\frac{m_s+1}{t+1}\right\rfloor,
\]
and choose a finite colored nerve $S_s$ of dimension at most $m_s$
and a subordinate map $\psi_s:X_{j,s}\to S_s$ as in
Lemma~\ref{lem:nerve}.  Thus every vertex $w$ of $S_s$ is assigned a
member $V_s(w)\in\alpha_{i,j,s}$ such that
\[
 \lambda_w(\psi_s(x))>0\quad\Longrightarrow\quad x\in V_s(w).
\]
Use the packet maps of Lemma~\ref{lem:packets} on these nerves.

For the $k$th block of stage $h$, set
\begin{align*}
 T_k&=\prod_{s=1}^{h_j}\prod_{a=1}^{m_{j,h}^{s,k}}S_s,
 &
 \tau_k(x)&=
 \bigl(\psi_s(\lambda_{j,h}^{s,k,a}(x))\bigr)_{s,a},\\
 C_k&=\sum_{s=1}^{h_j}m_{j,h}^{s,k}m_s,
 &
 N_k&=\sum_{s=1}^{h_j}m_{j,h}^{s,k}g_s.
\end{align*}
The space $T_k$ carries the finite product-face complex structure, and
$\dim\car(y)$ means the sum of the dimensions of the simplex carriers
of its factors.  The same convention applies to powers of $T_k$.

Suppose $d>2\delta>0$ and, for every $k$, there are an integer
$u_k\geq1$ and a continuous map
\[
 \mathcal R_k:X_{h,k}^{u_k}\longrightarrow T_k^{u_k}
\]
with the following properties.  In every individual $S_s$ factor,
the vertex support of $\mathcal R_k(x_1,\ldots,x_{u_k})$ is contained in the
vertex support of the corresponding factor of
$(\tau_k(x_1),\ldots,\tau_k(x_{u_k}))$.  Moreover,
\[
 \dim\car\bigl(\mathcal R_k(x_1,\ldots,x_{u_k})\bigr)
 \leq (d-2\delta)u_kn_{h,k}
 \qquad\text{for every }(x_1,\ldots,x_{u_k}).
\]
Then, for every $l\geq h$ and every block $v$ of stage $l$,
\begin{equation}\label{eq:ah-grouping-bound}
 \DD(\alpha_{i,l,v})
 \leq (d-2\delta)n_{l,v}+E,
 \qquad
 E=\sum_{k=1}^{h_h}(u_k-1)C_k.
\end{equation}
No injectivity of the connecting maps and no divergence of the
multiplicities of individual block types is required.
\end{lemma}

\begin{proof}
We first turn the pointwise carrier bound into a fixed finite target
complex.  Put
\[
 B_k=\left\lfloor(d-2\delta)u_kn_{h,k}\right\rfloor.
\]
Let $P_k$ be the union of all product-face cells of $T_k^{u_k}$ of
dimension at most $B_k$.  Every face of such a cell has no larger
dimension, so this union is a finite subcomplex.  Its dimension is at
most $B_k$, and the carrier assumption gives
$\mathcal R_k(X_{h,k}^{u_k})\subseteq P_k$.  The complex $P_k$ depends only on
the fixed stage-$h$ type and its chosen compression, not on any later
stage.  Likewise $T_k$ is a finite product-face complex with
$\dim T_k\leq C_k$.  These finite complexes and their products may be
triangulated without increasing dimension.

Fix $l\geq h$ and a target block $v$.  Put
$b_k=m_{h,l}^{k,v}$ and write
\[
 b_k=q_ku_k+r_k,\qquad q_k\geq0,\quad 0\leq r_k<u_k.
\]
Order the $b_k$ eigenmaps from $X_{l,v}$ to $X_{h,k}$ arbitrarily.
Separate them into $q_k$ ordered groups of $u_k$ maps and a remainder
of $r_k$ maps.  On every complete group use the continuous map
\[
 x\longmapsto
 \mathcal R_k\bigl(\lambda_1(x),\ldots,\lambda_{u_k}(x)\bigr)\in P_k.
\]
On each remainder occurrence use the unchanged map
$\tau_k\circ\lambda:X_{l,v}\to T_k$.  Taking the product of all
these maps gives a continuous map
\[
 \Phi_{l,v}:X_{l,v}\longrightarrow
 P_{l,v}:=
 \prod_{k=1}^{h_h}\bigl(P_k^{q_k}\times T_k^{r_k}\bigr).
\]
Factors with exponent zero are omitted.  This is a finite polyhedron,
and
\begin{align*}
 \dim P_{l,v}
 &\leq \sum_k q_k(d-2\delta)u_kn_{h,k}
       +\sum_k r_kC_k\\
 &\leq (d-2\delta)\sum_k b_kn_{h,k}
       +\sum_k (u_k-1)C_k\\
 &= (d-2\delta)n_{l,v}+E.
\end{align*}
The second inequality uses $d-2\delta>0$.  The last equality is the
unital rank identity
$n_{l,v}=\sum_k m_{h,l}^{k,v}n_{h,k}$.  Notice that the remainder
cost is bounded by the fixed constant $E$, even when a type occurs
fewer than $u_k$ times or does not occur at all.

It remains to verify refinement of the entire propagated cover.
Regard $P_{l,v}$ as a subcomplex of the product of the individual
$S_s$ factors corresponding to every occurrence of an eigenmap from
stage $j$ to stage $h$, and then from stage $h$ to the block $(l,v)$.
Choose a point $y\in\Phi_{l,v}(X_{l,v})$.  In each such $S_s$ factor,
choose a vertex $w$ whose barycentric coordinate at $y$ is positive.
The set on which all the finitely many chosen coordinates remain
positive is an open neighborhood $W_y$ of $y$ in $P_{l,v}$.

If $x\in\Phi_{l,v}^{-1}(W_y)$, support containment in every compressed
factor, and equality in every remainder factor, imply
\[
 \lambda_w\bigl(
 \psi_s(\lambda_{j,h}^{s,k,a}(\lambda_{h,l}^{k,v,b}(x)))
 \bigr)>0
\]
for its corresponding occurrence.  Hence
\[
 \lambda_{j,h}^{s,k,a}(\lambda_{h,l}^{k,v,b}(x))
 \in V_s(w).
\]
For each selected $V_s(w)\in\alpha_{i,j,s}$, choose its fixed
description as an intersection of inverse images of members of the
original component covers $\alpha_i$ under the eigenmaps from stage
$i$ to $X_{j,s}$.  Composing those eigenmaps with the two maps displayed
above, and intersecting the resulting assigned members over all
occurrences, produces one member $U_y$ of $\alpha_{i,l,v}$.
By construction,
\[
 \Phi_{l,v}^{-1}(W_y)\subseteq U_y.
\]
In particular, the assertion holds for the whole inverse image of the
target neighborhood, rather than merely for individual fibers or
individual input tuples.  Noninjective eigenmaps and repeated
eigenmaps cause no change to these inclusions.  The chosen composite
eigenmap list is the ordinary diagonal list for $\phi_{i,l}$, up to an
irrelevant permutation, so every occurrence required by the joined
cover is included.

Apply Lemma~\ref{lem:refinement} to $\Phi_{l,v}$ and the finite target
$P_{l,v}$ to obtain \eqref{eq:ah-grouping-bound}.  The construction
works separately for every later block with the same $E$, which
proves the asserted uniformity.
\end{proof}

\begin{remark}
The hypothesis on $\mathcal R_k$ is exactly the form supplied by
Lemma~\ref{lem:packets} after applying boundary-preserving compression
to all $N_k$ packets of the raw product nerve.  A type requiring no
compression can be included by taking $u_k=1$ and $\mathcal R_k=\tau_k$, provided
its raw dimension already satisfies
$C_k\leq(d-2\delta)n_{h,k}$.  The grouping argument uses compression on
the actual spaces $X_{h,k}^{u_k}$ and subsequently restricts along
tuples of later eigenmaps; it requires neither independent eigenmap
coordinates nor comparison over a Cartesian power of a building
space.
\end{remark}

\section{The reverse inequality}

\begin{theorem}\label{thm:lower}
Let $\mathcal P=(A_j,\phi_{j,j+1})$ be an ordinary globally diagonal
AH presentation as in Definition~\ref{def:ah}, with compact metrizable
bases and simple unital infinite-dimensional limit $A$.  For every finite
comparison bound $r>0$ for $A$,
\[
 \gamma(\mathcal P)\leq2r.
\]
Consequently,
\[
 \gamma(\mathcal P)\leq2\operatorname{rc}(A)
 \quad\text{in }[0,\infty].
\]
\end{theorem}

\begin{proof}
Fix a stage $i$ and a finite componentwise open cover $\alpha_i$.
In the notation of Definition~\ref{def:gamma}, put
\[
 d_j=\max_s
 \frac{\mathcal D(\alpha_{i,j,s})}{n_{j,s}},
 \qquad d=\lim_{j\to\infty}d_j.
\]
The sequence $d_j$ decreases, and each $d_j$ is finite.  It suffices to
show $d\leq2r$, since the supremum over such covers and then stages is
$\gamma(\mathcal P)$.  Suppose, to the contrary, that $d>2r$.

\paragraph{Parameters and primary packets.}
Choose an integer $p\geq1$ such that, with $t=2p+1$ and
$c=d/(t+1)$, one has $pc>r$.  This is possible because
$p/(2p+2)\to1/2$.  Choose $0<a<1/16$ so small that
\[
 G=(1-11a)pc-r>0.
\]
Fix
\[
 \varepsilon=\min\left\{\frac{ac}{4096},\frac{ad}{4096}\right\},
 \qquad \rho=c\left(1-\frac a{512}\right),
 \qquad \delta=\min\left\{\frac{ad}{2048},\frac{ac}{32}\right\}.
\]
All these constants are positive.  In particular, $d-2\delta>0$.

By convergence of $d_j$ and Lemma~\ref{lem:rank-growth}, choose a stage
$j\geq i$ such that $d_j\leq d+\varepsilon$ and
\begin{equation}\label{eq:ah-primary-minimum-rank}
 n_{\min}:=\min_s n_{j,s}\geq
 \max\left\{\frac{4096}{apc},\frac{4096t}{ad},
                         \frac4\varepsilon\right\}.
\end{equation}
For each component of $A_j$, abbreviate $n_s=n_{j,s}$ and set
\[
 \beta_s=\alpha_{i,j,s},\qquad
 m_s=\mathcal D(\beta_s),\qquad
 g_s=\left\lfloor\frac{m_s+1}{t+1}\right\rfloor,
 \qquad
 R_s=\left\lfloor(2-10a)pc\,n_s\right\rfloor.
\]
Choose a colored subordinate nerve map $\psi_s:X_{j,s}\to S_s$ of
dimension at most $m_s$, together with its $g_s$ packet maps
$F_{s,\nu}:X_{j,s}\to C_t$, as in Lemma~\ref{lem:packets}.
The possibility $g_s=0$ is allowed.

Only an upper bound on the primary packet densities is needed:
\begin{align*}
 \frac{g_s}{n_s}
 &\leq \frac{m_s+1}{(t+1)n_s}
 \leq c+\frac{\varepsilon}{t+1}
                  +\frac1{(t+1)n_s}
 \leq c+\frac{5\varepsilon}{4(t+1)}
 \leq c+\varepsilon.
\end{align*}
There is no assertion that individual components have a comparable
lower density.  The target ranks $R_s$ are proportional to $n_s$, rather
than to $g_s$.  This distinction supplies a tracial gap on every
component:
\begin{align*}
 \frac{R_s-pg_s}{n_s}
 &\geq(2-10a)pc-p(c+\varepsilon)-\frac1{n_s}\\
 &\geq(1-10a)pc-\frac{apc}{4096}-\frac{apc}{4096}\\
 &=r+G+apc\left(1-\frac2{4096}\right)>r+G>r.
\end{align*}
Here we used both the choice of $\varepsilon$ and
\eqref{eq:ah-primary-minimum-rank}.  Ranks in these formulas use the
unnormalized trace on matrix amplifications: a rank-$q$ projection over
the $s$th corner has normalized base-trace value $q/n_s$.

\paragraph{One witness and a high--low decomposition.}
Apply Lemma~\ref{lem:ah-embedding} to these packets and ranks.  It gives
one finite-stage comparison witness, and its stage may be advanced
without changing any primary data.  By Lemma~\ref{lem:rank-growth}, take
its stage $h\geq j$ so late that
\begin{equation}\label{eq:ah-witness-minimum-rank}
 \min_k n_{h,k}\geq\frac4{a\rho}.
\end{equation}
For the multiplicities $\mu_{s,k}=m_{j,h}^{s,k}$, put
\[
 N_k=\sum_s\mu_{s,k}g_s,\qquad
 C_k=\sum_s\mu_{s,k}m_s,\qquad
 q_k=\sum_s\mu_{s,k}R_s,\qquad
 v_k=\sum_s\mu_{s,k}.
\]
Unitality gives $n_{h,k}=\sum_s\mu_{s,k}n_s$.  Consequently
\begin{equation}\label{eq:ah-witness-block-budgets}
 \frac{C_k}{n_{h,k}}\leq d+\varepsilon,
 \qquad \frac{v_k}{n_{h,k}}\leq\frac1{n_{\min}},
 \qquad q_k\leq(2-10a)pc\,n_{h,k}.
\end{equation}
The definition of $g_s$ also gives $m_s\leq(t+1)g_s+t$, and hence
\begin{equation}\label{eq:ah-rounding-baseline}
 C_k\leq(t+1)N_k+t v_k.
\end{equation}

Let
\[
 T_k=\prod_s\prod_{\nu=1}^{\mu_{s,k}}S_s,
 \qquad
 \tau_k:X_{h,k}\longrightarrow T_k
\]
be the uncompressed product nerve map with coordinates
$\psi_s\circ\lambda_{j,h}^{s,k,\nu}$.
The baseline dimension bound for $T_k$ is $C_k$.
Call $k$ high if $N_k\geq\rho n_{h,k}$, and low otherwise.

For a low component, \eqref{eq:ah-rounding-baseline} gives
\begin{align*}
 \frac{C_k}{n_{h,k}}
 &<(t+1)\rho+\frac{t}{n_{\min}}\\
 &\leq d\left(1-\frac a{512}\right)+\frac{ad}{4096}
 =d-\frac{7ad}{4096}
 \leq d-\frac72\delta<d-2\delta.
\end{align*}
Thus for low $k$ take $u_k=1$ and $\mathcal R_k=\tau_k$.
No compression theorem is needed on these components.

For a high component, $N_k\geq\rho n_{h,k}\geq4/a>0$ by
\eqref{eq:ah-witness-minimum-rank}.  Lemma~\ref{lem:ah-embedding}
gives a single bundle monomorphism over the entire sparse-parameter
locus for its concatenated packet map $f_k:X_{h,k}\to C_t^{N_k}$,
with target rank $q_k+2p\lceil aN_k\rceil$.  Its ratio to the source
rank $pN_k$ is bounded by
\begin{align*}
 \frac{q_k+2p\lceil aN_k\rceil}{pN_k}
 &\leq \frac{2-10a}{1-a/512}+2a+\frac2{N_k}\\
 &\leq \frac{2-10a}{1-a/512}+\frac52a
 <2-7a<2.
\end{align*}
For the strict inequality, note that
\[
 \frac{2-10a}{1-a/512}
 =2-\frac{(10-1/256)a}{1-a/512}
 <2-(10-1/256)a.
\]
The imported torus compression theorem, Theorem~\ref{input:compression},
therefore supplies an integer $u_k\geq1$ and a continuous map
\[
 \mathcal G_k:X_{h,k}^{u_k}\longrightarrow C_t^{u_kN_k}
\]
with at least $(a/8)u_kN_k$ boundary slots at every point and exact
preservation of every original boundary slot.

Apply Lemma~\ref{lem:packets} to each primary nerve coordinate in each
of the $u_k$ copies, using its corresponding replacement packets from
$\mathcal G_k$.  Boundary preservation verifies the lemma's exact
relative hypothesis.  The resulting continuous map
\[
 \mathcal R_k:X_{h,k}^{u_k}\longrightarrow T_k^{u_k}
\]
has every vertex support contained in that of the corresponding original
coordinate of $\tau_k$.  All the slots here are real primary packets;
there are no padded or exceptional geometric slots.  Its product-cell
carrier dimension is at most
\[
 u_k C_k-\frac a8u_kN_k.
\]
After division by $u_kn_{h,k}$, this is at most
\begin{align*}
 d+\varepsilon-\frac a8\rho
 &\leq d-ac\left(\frac{1-a/512}{8}-\frac1{4096}\right)\\
 &<d-\frac{ac}{16}\leq d-2\delta.
\end{align*}
Indeed $a<1/16$ makes the coefficient in parentheses strictly greater
than $1/16$.  Thus both high and low components satisfy the common
carrier bound $(d-2\delta)u_kn_{h,k}$, with exact coordinatewise vertex
subordination.

\paragraph{All later destination components.}
Apply Lemma~\ref{lem:ah-grouping} to this finite family of maps and
exponents.  It yields the fixed finite remainder cost
\[
 E=\sum_k(u_k-1)C_k
\]
and, for every later stage $l\geq h$ and every destination component
$w$, the estimate
\begin{equation}\label{eq:ah-final-cover-estimate}
 \mathcal D(\alpha_{i,l,w})
 \leq(d-2\delta)n_{l,w}+E.
\end{equation}
This is a bound on every destination component, rather than on an
average of their normalized dimensions.  The arbitrary exponents
$u_k$ cause no restriction: copies of each component are grouped in
$u_k$-tuples, and fewer than $u_k$ copies of each type remain.

Choose $l$ so late that $\min_w n_{l,w}>2E/\delta$, using
Lemma~\ref{lem:rank-growth}.  Equation~\eqref{eq:ah-final-cover-estimate}
then gives
\[
 d_l=\max_w\frac{\mathcal D(\alpha_{i,l,w})}{n_{l,w}}
 <d-\frac32\delta<d.
\]
If $E=0$, the same strict conclusion holds at any sufficiently late
stage.  This contradicts $d_l\geq d$, since $d_l$ decreases to $d$.
The contradiction proves $d\leq2r$.

Finally take the suprema in Definition~\ref{def:gamma}.  This proves
$\gamma(\mathcal P)\leq2r$ for every finite positive comparison bound.
If such bounds exist, taking their infimum gives
$\gamma(\mathcal P)\leq2\operatorname{rc}(A)$; attainment of that
infimum is not required.  In particular, radius zero implies
$\gamma(\mathcal P)=0$.  If zero itself is a comparison bound, every
positive number is also a comparison bound, so the same conclusion
follows.  If there is no finite comparison bound, the right-hand side is
infinite and the inequality is automatic.  Conversely,
$\gamma(\mathcal P)=\infty$ excludes every finite comparison bound,
and hence forces $\operatorname{rc}(A)=\infty$.
\end{proof}

\section{The existing upper bound for noninjective presentations}

\begin{lemma}\label{lem:upper-bound}
For a presentation as in Theorem~\ref{thm:ah-main}, including a
presentation with noninjective connecting maps,
\[
 \operatorname{rc}(A)\leq\tfrac12\gamma(\mathcal P).
\]
\end{lemma}
\begin{proof}
We adapt the proof of \cite[Theorem~6.2]{NiuAH} to the
noninjective presentations allowed here by verifying the two
hypotheses of \cite[Lemma~6.1]{NiuAH}.  The finite-dimensional case follows
from Lemma~\ref{lem:finite-limit}, and the assertion is automatic if
$\gamma(\mathcal P)=\infty$.  Suppose therefore that $A$ is infinite
dimensional and $\gamma=\gamma(\mathcal P)<\infty$.

\emph{Finite-stage local approximants.}
Fix a finite subset $F\subset A$ and positive tolerances $\eta$ and
$\zeta$.  Approximate $F$ by the images of a finite set $F_i\subset A_i$
at some stage $i$.  Include the unit in $F_i$.  In the ambient matrix
realization of each source corner choose a componentwise finite open
cover $\alpha$ controlling, to an arbitrarily small tolerance, the
variation of these matrices and of the source unit projections.
The tolerance may depend on the finite maximum of the norms of the
chosen lifts; no norm bound on their images is substituted for that
maximum.
By Definition~\ref{def:gamma} and cover-rate monotonicity, a
sufficiently late stage $j$ satisfies
\[
 \max_k\frac{\mathcal D(\alpha_{i,j,k})}{n_{j,k}}
     <\gamma+\zeta.
\]

Carry out the construction in
\cite[Proof of Theorem~4.3]{NiuAH} inside this finite stage, before
mapping it to $A$.  For clarity, for each target block use its
diagonal corner realization
\[
 E_{j,k}\cong\bigoplus_{s,a}
                   (\lambda_{i,j}^{s,k,a})^*E_{i,s}
 \subset X_{j,k}\times\mathbb C^{R_k},
 \qquad R_k=\sum_s m_{i,j}^{s,k}r_{i,s}.
\]
In these coordinates both the images of $F_i$ and the target unit
projection have small matrix variation on every member of
$\alpha_{i,j,k}$.  This statement concerns the diagonal realization;
it does not assert small variation after a varying bundle unitary
has transported the matrices to a different realization.

Refine $\alpha_{i,j,k}$ by a cover of order
$\mathcal D(\alpha_{i,j,k})$, and use a subordinate partition of
unity to map $X_{j,k}$ into its finite nerve $\Delta_k$.
Writing $\Omega_k$ for the compact image of this map, pullback of
matrix functions gives a subalgebra
$M_{R_k}(C(\Omega_k))\subset M_{R_k}(C(X_{j,k}))$.
The partition-of-unity averages of the finitely many matrices
belong to this subalgebra and approximate those matrices in norm.
The average of the target unit projection can be cut spectrally at
$1/2$ to obtain a projection $e_k$ in this subalgebra arbitrarily
close to the target unit projection.  A unitary close to the identity
in the \emph{full ambient matrix algebra} identifies these two
projections.  Conjugating the $e_k$-corner by that unitary, and then
transporting it back by the prescribed bundle identification, gives
a unital subalgebra $C_j\subset A_j$ with
\[
 C_j\cong\bigoplus_k
       e_kM_{R_k}(C(\Omega_k))e_k,
 \qquad \operatorname{rank}(e_k)=n_{j,k},
 \qquad \dim\Omega_k\leq\mathcal D(\alpha_{i,j,k}).
\]
It approximates $\phi_{i,j}(F_i)$ as closely as desired.  The rank
statement holds at every point of $\Omega_k$, since every such
point is the image of a point of $X_{j,k}$ and the pulled-back
projection is unitarily equivalent to the unit of rank $n_{j,k}$.
The original approximation of $F$ can therefore be chosen so that
the image $C=\phi_{j,\infty}(C_j)$ approximates $F$ within $\eta$.

\emph{Quotient images do not increase the dimension ratio.}
The homomorphism $C_j\to C$ need not be injective.  Every ideal in
a homogeneous corner corresponds to an open subset of its spectrum.
Consequently its quotient is the same corner restricted to a closed
subset.  Omitting zero summands, there are compact sets
$K_k\subset\Omega_k$ with
\[
 C\cong\bigoplus_k
       (e_k|_{K_k})M_{R_k}(C(K_k))(e_k|_{K_k}).
\]
This is a unital subalgebra of $A$.  Its surviving unit ranks are
still $n_{j,k}$ and
$\dim K_k\leq\dim\Omega_k$.  The homogeneous-corner comparison
estimate of \cite[Corollary~5.2]{TomsComparison}, also used in
\cite[Proof of Theorem~6.2]{NiuAH}, gives
\[
 \operatorname{rc}(C)
 \leq\max_k\frac{\dim K_k}{2n_{j,k}}
 <\frac\gamma2+\frac\zeta2.
\]
For zero-dimensional spectra the same bound is zero, since the
homogeneous corner is AF.  Thus the local approximation condition
of \cite[Lemma~6.1]{NiuAH} holds with $r=\gamma/2$.  The quotient
spectra need not be connected; the displayed comparison estimate
uses their finite covering dimension and their constant unit ranks.

\emph{Projections with nearly constant normalized traces.}
Fix one source stage $i$ and put
$N_i=\max_s n_{i,s}$.  For any $u\in[0,1]$ and any late stage $j$,
order the eigenmap occurrences in each target block.  The corresponding
pullback unit bundles have positive ranks at most $N_i$ and total
rank $n_{j,k}$.  Select a prefix of these occurrences whose total
rank is the largest prefix rank not exceeding $u n_{j,k}$.
Its rank differs from $u n_{j,k}$ by at most $N_i$; for $u=1$ select
all occurrences.  Projection onto the selected summands, transported
by the global diagonal bundle identification, is a continuous
projection $z_j\in A_j$.  Uniformly over all normalized fiber traces,
\[
 \left|\tau_x(z_j)-u\right|
 \leq\frac{N_i}{\min_k n_{j,k}}
 \longrightarrow0
\]
by Lemma~\ref{lem:rank-growth}.  The same estimate holds for every
trace on $A$ after mapping $z_j$ to the limit, because its restriction
to $A_j$ is a convex integral of these fiber traces.  Adding copies
of the unit gives the corresponding approximation to any
nonnegative constant in matrix amplifications.  The same finite-stage local construction applies to matrix
amplifications, with unit ranks multiplied by the amplification size.
These estimates and projections verify the hypotheses of
\cite[Lemma~6.1]{NiuAH} with the trace normalization fixed in
Section~\ref{sec:common-conventions}.

Finally $A$ is simple and exact.  The two verified conditions allow
\cite[Lemma~6.1]{NiuAH} to pass the local comparison bounds to the
limit, giving $\operatorname{rc}(A)\leq\gamma/2$.
\end{proof}

\section{Consequences for diagonal AH algebras}\label{sec:ah-consequences}
\begin{proof}[Proof of Theorem~\ref{thm:ah-main}]
If $A$ is finite dimensional, simplicity identifies it with a full matrix
algebra, and Lemma~\ref{lem:finite-limit} gives both sides equal to zero.
Otherwise the reverse inequality is Theorem~\ref{thm:lower}. Niu's upper bound,
with the noninjective case explained in Lemma~\ref{lem:upper-bound},
gives $\rc(A)\leq\gamma(\mathcal P)/2$.
Together these yield the claimed equality, including infinite values.
\end{proof}

\begin{corollary}\label{cor:invariance}
Within the class of presentations in Theorem~\ref{thm:ah-main},
$\gamma(\mathcal P)$ is independent of the presentation. In particular,
$\rc(A)=0$ if and only if $\gamma(\mathcal P)=0$.
\end{corollary}
\begin{proof}
The equality identifies $\gamma(\mathcal P)$ with $2\rc(A)$.
\end{proof}

The proof uses simplicity to obtain uniform growth of the minimum matrix
rank and to invoke the known upper bound. Compact metrizability permits
the application of the compression theorem. Ordinary diagonal maps
provide distinct coordinates for every eigenmap occurrence and preserve
the global trivial frames in matrix amplifications. No lower bound on the
multiplicity of any fixed component type is used: the bounded remainders
from all types are absorbed by the growing total matrix rank.


\part{Amenable actions with URP}\label{part:actions}

\section{Amenable actions, URP, and COS}\label{sec:actions-conventions}
\subsection{Actions and COS}
We use a left action, written $T_gx=gx$, with $T_gT_h=T_{gh}$. On $C(X)$ we set $\alpha_g(f)(x)=f(g^{-1}x)$, and write $A=C(X)\rtimes_\alpha\Gamma$. Amenability makes the full and reduced crossed products agree. The algebra is unital and nuclear. Freeness identifies its normalized traces with invariant probability
measures: $\tau_\mu=\mu\circ\mathbb E$, where
$\mathbb E:A\to C(X)$ is the canonical expectation.
Indeed, for $g\ne e$, freeness gives a finite open cover by sets $V$
with $V\cap gV=\varnothing$. A partition of unity on this cover and
traciality imply $\tau(fu_g)=0$ for $f\in C(X)$. The restriction of
a trace to $C(X)$ is invariant, and conversely $\mu\circ\mathbb E$
is a trace for every invariant probability measure $\mu$.

For an open set $O\subset X$, choose $h_O\in C(X)_+$ with
$\{h_O>0\}=O$. In the sense of \cite[Definition 4.1]{Niu}, the action
has COS if there are $\lambda\in(0,1]$ and $k\in\N$ such that, for
all open $O_1,O_2\subset X$,
\[
 \mu(O_1)<\lambda\mu(O_2)\quad(\mu\in\Mu)
 \quad\Longrightarrow\quad h_{O_1}\bpre h_{O_2}^{\oplus k}
 \quad\text{in }M_\infty(A).
\]
The Cuntz classes involved do not depend on the choices of $h_O$.

\subsection{Finite sets and cover rates}
For finite $K,F\subset\Gamma$ containing the identity in $K$, say that $F$ is $(K,\eps)$-invariant if $|KF\setminus F|<\eps|F|$. Define its $K$-core and inner collar by
\[
 F_K=\{v\in F:Kv\subset F\},\qquad \partial_KF=F\setminus F_K.
\]
Then $|\partial_KF|\leq\sum_{k\in K}|kF\setminus F|$. We always request enough invariance to control the inner collar itself; a union bound may require dividing the tolerance by $|K|$.

For a finite open cover $\UU$, write
\[
 \UU^F=\bigvee_{v\in F}T_v^{-1}\UU,
 \qquad m(F)=\DD(\UU^F).
\]
Put $m(\varnothing)=0$ and $m_0=\DD(\UU)$. We have
\[
 m(Fg)=m(F),\quad m(F\cup L)\leq m(F)+m(L),\quad
 0\leq m(F)\leq m_0|F|.
\]
The Ornstein--Weiss lemma gives a common limit $d_\UU$ of $m(F_n)/|F_n|$ along every left F\o lner sequence; see \cite{LW,Krieger}. The mean dimension is $\sup_\UU d_\UU$.

\begin{lemma}[Uniform form of the cover limit]\label{lem:coverlimit}
For every $\zeta>0$, there are a finite $K$ and $\eps>0$ such that every nonempty $(K,\eps)$-invariant $F$ satisfies
\[
 \left|m(F)/|F|-d_\UU\right|<\zeta.
\]
\end{lemma}
\begin{proof}
Otherwise exhaust $\Gamma$ by finite sets $K_n$ containing the identity, and choose $(K_n,1/n)$-invariant counterexamples $F_n$. These form a left F\o lner sequence, contradicting the common limit. In an infinite group their cardinalities tend to infinity.
\end{proof}

\begin{lemma}[Uniform closed-set frequency]\label{lem:frequency}
If $C\subset X$ is closed and $b>\sup_{\mu\in\Mu}\mu(C)$, then for all sufficiently invariant finite $F$, uniformly in $x\in X$,
\[
 |\{v\in F:vx\in C\}|\leq b|F|.
\]
\end{lemma}
\begin{proof}
If the conclusion failed, choose a F\o lner sequence of counterexample windows and points $x_n$. Any weak-star cluster point of the empirical measures $|F_n|^{-1}\sum_{v\in F_n}\delta_{vx_n}$ is invariant. The closed-set inequality in the Portmanteau theorem would give it mass at least $b$ on $C$, a contradiction. The same sequential argument as in Lemma~\ref{lem:coverlimit} makes the conclusion uniform over all sufficiently invariant finite sets.
\end{proof}

\subsection{URP and exact tilings}
For $E\subset X$, define its orbit capacity by
\[
 \operatorname{ocap}(E)=
 \lim_F\sup_{x\in X}\frac1{|F|}\sum_{v\in F}1_E(vx),
\]
where the limit is taken along left F\o lner sets. The limit exists
by the Ornstein--Weiss lemma applied to the invariant subadditive
counting function. For every Borel $E$ and $\mu\in\Mu$, one has
$\mu(E)\leq\operatorname{ocap}(E)$.

We use the finite-castle formulation of URP in
\cite[Definition 3.1 and Lemma 3.2]{Niu}. For every finite
$K\subset\Gamma$ containing $e$ and every $\varepsilon,\delta>0$,
there are finitely many closed bases $V_i$ and nonempty finite
$(K,\varepsilon)$-invariant shapes $S_i$ such that the levels
$T_vV_i$, $v\in S_i$, are pairwise disjoint and
\[
 \operatorname{ocap}\!\left(
 X\setminus\bigsqcup_i\bigsqcup_{v\in S_i}T_vV_i
 \right)<\delta.
\]
The remainder consequently has uniformly small invariant-measure
mass. Lemma~\ref{lem:closed-to-open-castle} permits enlargement to
open bases while preserving the finite disjointness relations.
The right-action convention of \cite{Niu} converts to ours by
$gx=xg^{-1}$ and inversion of tower shapes; right invariance of
the original shapes becomes left invariance here.

An exact finite-shape tiling of $\Gamma$ is a partition into right translates $F_sc$, using finitely many finite shapes $F_s$. The centers belong to the tiles, equivalently $e\in F_s$. By \cite[Theorem 4.3]{DHZ}, the shapes can all be made arbitrarily invariant. They can also all be made arbitrarily large: if $F_s$ is $(K,\eps)$-invariant and nonempty, then $|K|\leq |KF_s|<(1+\eps)|F_s|$. Right normalization of shapes preserves left invariance and the quantities $m(F_s)$.

\section{Finite-shape allocation from URP}
\begin{lemma}[Open castles]
\label{lem:closed-to-open-castle}
Suppose $(S_i,A_i)_{i=1}^m$ is a finite castle with closed bases:
all the sets $sA_i$, $i\leq m$, $s\in S_i$, are pairwise disjoint.
There are open sets $V_i\supset A_i$ such that the sets $sV_i$
remain pairwise disjoint.  In particular, any lower bound on the
invariant-measure mass of the closed castle remains valid for this
open castle, and its shapes do not change.
\end{lemma}

\begin{proof}
The finitely many disjoint compact sets $sA_i$ have pairwise
disjoint open neighborhoods $W_{i,s}$.  Set
\[
  V_i=\bigcap_{s\in S_i}s^{-1}W_{i,s}.
\]
Then $V_i$ is open, contains $A_i$, and $sV_i\subset W_{i,s}$.
No commutation of group elements is used.
\end{proof}

Fix an exact discrete tiling
\begin{equation}
\label{eq:fixed-primary-group-tiling}
 \Gamma=\bigsqcup_{s=1}^q\bigsqcup_{c\in\mathcal C_s}F_sc,
 \qquad e\in F_s,
\end{equation}
with finitely many nonempty finite shapes.  Put
\[
 H=\bigcup_{s=1}^qF_sF_s^{-1}.
\]
This set is finite, symmetric, and contains $e$.  If a tile $F_sc$
meets a set $Y$, then
\begin{equation}
\label{eq:tile-meeting-enlargement}
  F_sc\subset HY.
\end{equation}
Indeed, if $fc\in Y$ and $f'\in F_s$, then
$f'c=(f'f^{-1})(fc)$.

\begin{proposition}[Finite-shape allocation]
\label{prop:finite-shape-allocation}
Suppose the action has the uniform Rokhlin property.  Fix the
tiling \eqref{eq:fixed-primary-group-tiling}, finite nonempty macro
shapes $E_1,\ldots,E_L$ containing $e$, and $0<\eta<1$ such that
\[
  |\partial_HE_l|\leq\eta|E_l|\qquad(1\leq l\leq L).
\]
Let $\alpha>0$ and $\theta_1,\ldots,\theta_L>0$.  There are open
sets $B_1,\ldots,B_q\subset X$, compact sets $K_s\subset B_s$,
and, for each $l$, a finite family of pairwise disjoint open sets
$(O_{l,\lambda})_{\lambda\in\Lambda_l}$, a continuous function
$w_l:X\to[0,1]$, and a closed set $C_l\subset X$, satisfying the
following properties.
\begin{enumerate}
\item The sets $fB_s$, $s\leq q$, $f\in F_s$, are pairwise
disjoint, and
\[
  \sum_s|F_s|\mu(B_s)>1-\alpha
  \qquad(\mu\in M_\Gamma(X)).
\]
Empty bases are allowed.
\item Writing $O_l=\bigsqcup_{\lambda\in\Lambda_l}O_{l,\lambda}$,
one has
\[
 \operatorname{supp}(w_l)\subset O_l,
 \qquad \{w_l<1\}\subset C_l,
 \qquad \sup_{\mu\in M_\Gamma(X)}\mu(C_l)<\theta_l/2.
\]
In particular, the label $\lambda$ is locally constant on a
neighborhood of $\operatorname{supp}(w_l)$.
\item Each $\lambda\in\Lambda_l$ specifies a finite pattern
$\mathcal P_{l,\lambda}$ of pairs $(s,d)$ such that the right
translates $F_sd$, $(s,d)\in\mathcal P_{l,\lambda}$, are pairwise
disjoint subsets of $E_l$.  Their hole set
\[
  J_{l,\lambda}
  =E_l\setminus\bigsqcup_{(s,d)\in\mathcal P_{l,\lambda}}F_sd
\]
satisfies the pointwise estimate
\[
  J_{l,\lambda}\subset\partial_HE_l,
  \qquad |J_{l,\lambda}|\leq\eta|E_l|.
\]
For every $(s,d)\in\mathcal P_{l,\lambda}$,
\[
 dO_{l,\lambda}\subset B_s,
 \qquad
 d\bigl(\operatorname{supp}(w_l)\cap O_{l,\lambda}\bigr)
 \subset K_s.
\]
\item There exist a finite symmetric $K_0\subset\Gamma$, containing
$e$, and $\delta_0>0$ such that for every $(K_0,\delta_0)$-invariant finite $F$, every $x\in X$, and every
$l$,
\[
 |\{g\in F:w_l(gx)<1\}|<\theta_l|F|.
\]
The allocation data are fixed before the threshold pair
$(K_0,\delta_0)$ is chosen.
\end{enumerate}
\end{proposition}

\begin{proof}
Choose
\[
 0<\varepsilon<\tfrac18\min\{1,\alpha,\theta_1,\ldots,\theta_L\}.
\]
Let $Q$ be a finite symmetric set containing
$H\cup\bigcup_lHE_l$. Apply URP with invariance tolerance
$\varepsilon/|Q|$ and remainder tolerance $\varepsilon$.
The inner-collar estimate of Section~\ref{sec:actions-conventions} and
Lemma~\ref{lem:closed-to-open-castle} give an open castle
$(S_i,V_i)_{i=1}^m$ such that
\[
 |\partial_QS_i|<\varepsilon|S_i|,
 \qquad
 \sum_i|S_i|\mu(V_i)>1-\varepsilon
 \quad(\mu\in M_\Gamma(X)).
\]
Thus each shape is simultaneously strongly invariant for every
finite window specified before this application of URP.

For each $i,s$ retain precisely the primary group tiles wholly
contained in $S_i$:
\[
 \mathcal C_{i,s}=\{c\in\mathcal C_s:F_sc\subset S_i\},
 \qquad
 B_s=\bigcup_i\bigcup_{c\in\mathcal C_{i,s}}cV_i.
\]
These unions are finite because $e\in F_s$ implies $c\in S_i$.
All their primary levels are original castle levels, and different
retained tiles use different levels.  Hence the sets $fB_s$ are
pairwise disjoint.  If $t\in S_i\setminus\partial_HS_i$, the unique
primary group tile containing $t$ lies in $Ht\subset S_i$, by
\eqref{eq:tile-meeting-enlargement}.  The retained tiles therefore
cover more than $(1-\varepsilon)|S_i|$ positions of $S_i$.  By
invariance and disjointness,
\[
 \sum_s|F_s|\mu(B_s)
 \geq(1-\varepsilon)\sum_i|S_i|\mu(V_i)
 >(1-\varepsilon)^2>1-\alpha.
\]

Fix $l$.  For each $i$ let
\[
 I_{i,l}=\{v\in S_i:HE_lv\subset S_i\}.
\]
Since $HE_l\subset Q$, this set has cardinality greater than
$(1-\varepsilon)|S_i|$.  Use the labels
$\lambda=(i,v)$ for $v\in I_{i,l}$, and set
\[
 O_{l,(i,v)}=vV_i,
 \qquad
 \mathcal P_{l,(i,v)}
 =\{(s,cv^{-1}):c\in\mathcal C_s,\ F_sc\subset E_lv\}.
\]
The labeled regions are open and pairwise disjoint, and
\[
 \mu(O_l)>(1-\varepsilon)^2>1-2\varepsilon.
\]
Each pair in the pattern has $d=cv^{-1}$, so
\[
 F_sd=F_scv^{-1}\subset E_l,
 \qquad
 d(vV_i)=cV_i\subset B_s.
\]
The inclusion in $B_s$ follows because $F_sc\subset E_lv\subset S_i$.
The tiles in the pattern are disjoint, being right translates of
members of the fixed exact tiling.

To verify the hole bound, let $t\in E_l$ be a hole.  The unique
primary tile containing $tv$ cannot lie wholly in $E_lv$, for then
it would occur in the pattern.  Choose another point $t'$ of this
tile outside $E_lv$.  Equation~\eqref{eq:tile-meeting-enlargement}
gives $t'=htv$ for some $h\in H$, whence $ht\notin E_l$.
Thus $t\in\partial_HE_l$.  Notice that the stronger condition
$HE_lv\subset S_i$ also guarantees that every primary group tile
meeting $E_lv$ was retained in slicing $S_i$.  These statements
are entirely within a fixed group tiling and use no equivariant
choice of a global tiling on $X$.

We now choose the weights uniformly over invariant measures.
There are continuous functions $\varphi_n:X\to[0,1]$, increasing
pointwise to $1_{O_l}$, with compact supports in $O_l$.  If $O_l=X$
use the constant function one; otherwise cutoff functions of
$\operatorname{dist}(x,X\setminus O_l)$ have these properties.
Since $\mu(O_l)>1-2\varepsilon>1-\theta_l/4$, the increasing open
subsets
\[
 \left\{\mu\in M_\Gamma(X):
       \int_X\varphi_n\,d\mu>1-\theta_l/2\right\}
\]
cover the compact space $M_\Gamma(X)$.  One of them is the whole
space.  Fix its function $\varphi$, set $U_l=\{\varphi>0\}$, and
put $C_l=X\setminus U_l$.  The continuous integral function has
minimum strictly greater than $1-\theta_l/2$, so
\[
 \sup_\mu\mu(C_l)
 \leq 1-\inf_\mu\int_X\varphi\,d\mu<\theta_l/2.
\]
Choose $w_l$ equal to one on the compact set
$\operatorname{supp}(\varphi)$ and with support compactly contained
in $O_l$.  This is possible since that compact set is contained
in the open set $O_l$.  In particular $w_l=1$ on $U_l$, giving
$\{w_l<1\}\subset C_l$.

The finite open partition $(O_{l,\lambda})_\lambda$ of $O_l$
restricts to a finite clopen partition of $\operatorname{supp}(w_l)$.
Consequently every intersection
$\operatorname{supp}(w_l)\cap O_{l,\lambda}$ is compact.  Define
\[
 K_s=\bigcup_{l,\lambda}
     \bigcup_{(s,d)\in\mathcal P_{l,\lambda}}
     d\bigl(\operatorname{supp}(w_l)\cap O_{l,\lambda}\bigr).
\]
It is a finite union of compact subsets of $B_s$, and hence is
compactly contained in $B_s$.  Finally apply
Lemma~\ref{lem:frequency} to the finitely many
sets $C_l$ with bounds $3\theta_l/4$, which exceed their uniform
measure bounds.  Their visit counts are at most
$(3\theta_l/4)|F|<\theta_l|F|$. Enlarge the finitely many test
sets to one symmetric set containing $e$ and decrease their
tolerances to one common positive tolerance.
The asserted strict frequency estimate follows from
$\{w_l<1\}\subset C_l$.
\end{proof}

\begin{corollary}[Weighted exceptional macro centers]
\label{cor:weighted-bad-macro-centers}
In Proposition~\ref{prop:finite-shape-allocation}, let $a_l\geq0$
be fixed weights.  Suppose $F$ is beyond the uniform frequency
threshold, and let $\mathcal D_l$ be finite sets of centers such
that the macro tiles $E_lc$, $c\in\mathcal D_l$, are pairwise
disjoint subsets of $F$.  Then, for every $x\in X$,
\[
 \sum_l a_l
    |\{c\in\mathcal D_l:w_l(cx)<1\}|
 \leq |F|\sum_l a_l\theta_l.
\]
Thus the weighted exceptional-center budget can be made
arbitrarily small by prescribing the $\theta_l$ after the finite
macro shapes and weights have been fixed.
\end{corollary}

\begin{proof}
Since $e\in E_l$ and $E_lc\subset F$, every center belongs to $F$.
Moreover $w_l(cx)<1$ implies $cx\in C_l$.  The sum is therefore
bounded by
\[
 \sum_l a_l\sum_{g\in F}1_{C_l}(gx)
 <|F|\sum_l a_l\theta_l
\]
unless all positive weights vanish, when the conclusion is
immediate.  Disjointness is needed elsewhere for the volume and
slot counts, not for this frequency estimate.
\end{proof}

\begin{remark}[Orientation and order of choices]
\label{rem:allocation-orientation-order}
The placement attached to a label $(i,v)$ is $d=cv^{-1}$, rather
than $v^{-1}c$.  It satisfies $d(vy)=cy$ and sends the labeled
region directly into the primary base.  This equality is the
reason no abelian or finite-conjugacy assumption occurs.
The fixed primary tiling and all the finite macro shapes are
chosen before the small-error URP castle.  The closed exceptional
sets and their uniform frequency thresholds are obtained after
that castle and its compact cutoffs.  Any later finite-propagation
witness may enlarge the invariance requirements for the final
orbit window without changing this allocation.
\end{remark}

\section{Comparison and parametrized bundle embeddings}
\label{sec:comparison-bundle}

This section proves the analytic step independently of the construction of
the allocation.  In particular, the allocation assumptions in
Proposition~\ref{prop:allocated-bundle-embedding} are hypotheses of that
proposition.  The later allocation argument must establish them with the
stated uniform estimates.

We use a left action $T:\Gamma\curvearrowright X$, so that
$T_gT_h=T_{gh}$, and put
\[
 \alpha_g(f)(x)=f(T_{g^{-1}}x),\qquad
 A=C(X)\rtimes_{\alpha,r}\Gamma.
\]
For $x\in X$, the orbit representation on $\ell^2(\Gamma)$ is
\begin{equation}
\label{eq:orbit-representation-convention}
 (\pi_x(f)\xi)(v)=f(T_vx)\xi(v),\qquad
 (\pi_x(u_g)\xi)(v)=\xi(g^{-1}v).
\end{equation}
These formulas are covariant: the diagonal entry of
$\pi_x(u_g)\pi_x(f)\pi_x(u_g)^*$ at $v$ is
$f(T_{g^{-1}v}x)=\alpha_g(f)(T_vx)$.  They also satisfy
$\pi_x(u_g)\pi_x(u_h)=\pi_x(u_{gh})$.  Thus a Fourier polynomial
$\sum_{g\in K}A_gu_g$, applied to vectors supported in a finite set
$F\subset\Gamma$, has output supported in $KF$.  A tower with shape
$E$ and a return at $v$ has footprint $Ev$.  These are right translates
of tower shapes; both conventions will be retained below.

Fix $p\geq1$ and write $t=2p+1$.  We use the framed torus
$L\subset(-1/32,1/32)^t$ and its rank-$p$ complex bundle
$\mathcal H\subset L\times\mathbb C^{2p}$ from the external torus-compression
theorem (Theorem~\ref{input:compression}).  Let $Q:L\to M_{2p}(\mathbb C)$ be its orthogonal projection.
For $0<a<1/16$ and $n\geq1$, write
\[
 Z_{an}(t,n)=\{z\in([-2,2]^t)^n:
              \#\{\nu:z_\nu\neq0\}\leq an\}.
\]
The support count here concerns whole $t$-dimensional slots.

\begin{proposition}[An allocated comparison witness gives one bundle embedding]
\label{prop:allocated-bundle-embedding}
Suppose that $r\geq0$ is a comparison bound for $A$, with normalized
traces and unnormalized matrix traces.  Let $s=1,\ldots,S$ index a
finite open castle with nonempty finite shapes $F_s\subset\Gamma$ and
bases $B_s\subset X$: all the sets $T_aB_s$, for $a\in F_s$ and
$1\leq s\leq S$, are pairwise disjoint.

For each $s$, fix positive integers $g_s,R_s$, a compact set
$K_s\subset B_s$, and continuous maps
\[
 F_{s,j}:X\longrightarrow[-1,1]^t,\qquad 1\leq j\leq g_s.
\]
Assume that
\begin{equation}
\label{eq:castle-tracial-rank-gap}
 \sum_{s=1}^S(R_s-pg_s)\mu(B_s)>r
 \quad\text{for every invariant probability measure }\mu.
\end{equation}

Then there is a finite set $K\subset\Gamma$ with the following
property.  Let $F\subset\Gamma$ be any nonempty finite window, and
suppose that there are $n\geq1$ continuous maps
$f_\nu:X\to[-1,1]^t$, continuous weights
$w_\nu:X\to[0,1]$, and locally constant maps
\[
 \ell_\nu:\{w_\nu>0\}\longrightarrow
 \{(s,v,j):1\leq s\leq S,\ v\in F,\ 1\leq j\leq g_s\}
 \qquad(1\leq\nu\leq n)
\]
such that:
\begin{enumerate}
\item whenever $\ell_\nu(x)=(s,v,j)$,
\begin{equation}
\label{eq:allocated-slot-agreement}
 T_vx\in K_s,\qquad f_\nu(x)=F_{s,j}(T_vx);
\end{equation}
\item for every $x$, the labels $\ell_\nu(x)$ of the slots with
$w_\nu(x)>0$ are pairwise distinct;
\item for some $\beta\geq0$,
\begin{equation}
\label{eq:allocated-slot-exception-count}
 \#\{\nu:w_\nu(x)<1\}\leq\beta n\qquad(x\in X).
\end{equation}
\end{enumerate}
Let $q_F$ be an integer such that
\begin{equation}
\label{eq:allocated-output-rank-count}
 \sum_{s=1}^S R_s\#\{v\in KF:T_vx\in B_s\}\leq q_F
 \qquad(x\in X).
\end{equation}
Put $f=(f_1,\ldots,f_n)$ and
\[
 \mathcal T_F=\{(x,z)\in X\times Z_{an}(t,n):
                  f(x)-z\in L^n\}.
\]
In place of Equation~\eqref{eq:allocated-slot-exception-count}, it
suffices to assume the weaker condition
\begin{equation}
\label{eq:allocated-slot-exception-count-on-locus}
 \#\{\nu:z_\nu=0,\ w_\nu(x)<1\}\leq\beta n
 \qquad((x,z)\in\mathcal T_F).
\end{equation}
The rank-$pn$ bundle
\[
 \mathcal E_F=\bigoplus_{\nu=1}^n
           \mathcal H_{f_\nu(x)-z_\nu}
 \longrightarrow\mathcal T_F
\]
admits one continuous bundle monomorphism into the trivial bundle of
rank
\begin{equation}
\label{eq:allocated-bundle-target-rank}
 q_F+2p\lceil(a+\beta)n\rceil.
\end{equation}
In particular, the hypothesis of the external torus-compression
theorem holds whenever this integer is strictly less than $2pn$.

One may take
\begin{equation}
\label{eq:castle-output-packing-bound}
 q_F=\left\lfloor
       \max_{1\leq s\leq S}\frac{R_s}{|F_s|}\,
       \left|\left(\bigcup_{s=1}^SF_s\right)KF\right|
      \right\rfloor.
\end{equation}
The cutoffs, Fourier witness, and its propagation set $K$ are chosen
before $F$, $n$, and the slot allocation.
\end{proposition}

\begin{proof}
\emph{Positive functions and one comparison witness.}
For each $s$, choose continuous functions
$\rho_{0,s},\rho_{1,s}:X\to[0,1]$ such that
\[
 \rho_{0,s}|_{K_s}=1,\qquad
 \operatorname{supp}(\rho_{0,s})\subset B_s,
 \qquad \{\rho_{1,s}>0\}=B_s.
\]
The first choice uses compact containment.  The second can be obtained
from the distance to the complement of $B_s$; if $B_s=X$, take the
constant function one.

Choose a tubular neighborhood $V$ of $L$ in $(-1,1)^t$, a retraction
$\pi_L:V\to L$, and a continuous function $\psi:[-1,1]^t\to[0,1]$
with compact support in $V$ and equal to one on a neighborhood of $L$.
Put $\widetilde Q=Q\circ\pi_L$ on $V$, and extend
$\psi\widetilde Q$ by zero.  For some $\delta>0$, compactness gives
\begin{equation}
\label{eq:castle-nearby-projection}
 y\in L,\ \|u-y\|_\infty<\delta
 \ \Longrightarrow
 u\in V,\quad\psi(u)=1,\quad
 \|\widetilde Q(u)-Q(y)\|<1/2.
\end{equation}
Consequently $\widetilde Q(u)$ restricts to an injective map from
$\mathcal H_y$ into $\operatorname{ran}\widetilde Q(u)$.

Write $D_0=2p\sum_sg_s$ and $D_1=\sum_sR_s$, and define
\begin{align}
 P_0(x)&=\bigoplus_{s=1}^S\bigoplus_{j=1}^{g_s}
       \rho_{0,s}(x)\psi(F_{s,j}(x))
                 \widetilde Q(F_{s,j}(x)),
       \label{eq:castle-source-positive}\\
 P_1(x)&=\bigoplus_{s=1}^S\rho_{1,s}(x)1_{R_s}.
       \label{eq:castle-target-positive}
\end{align}
Thus $P_0\in M_{D_0}(C(X))_+$ and
$P_1\in M_{D_1}(C(X))_+$.  For a trace $\tau$ on $A$, its restriction
to $C(X)$ is integration against an invariant probability measure
$\mu$.  For positive matrix functions, functional calculus and bounded
convergence give
\[
 d_\tau(P)=\int_X\operatorname{rank}P(x)\,d\mu(x).
\]
It follows that
\[
 d_\tau(P_0)+r
 \leq\sum_s pg_s\mu(B_s)+r
 <\sum_sR_s\mu(B_s)=d_\tau(P_1).
\]
By comparison, $P_0\precsim_A P_1$.  Taking a rectangular corner of a
Cuntz witness and approximating its entries by Fourier polynomials
gives
\begin{equation}
\label{eq:castle-finite-propagation-witness}
 T_0=\sum_{g\in K}A_gu_g\in M_{D_1,D_0}(A),\qquad
 \|T_0^*P_1T_0-P_0\|<1/2,
\end{equation}
where $K$ is finite and each $A_g$ is a continuous matrix function.
All objects in this paragraph are independent of the eventual window
and slot allocation.

\emph{A finite matrix with simultaneous injectivity.}
Restrict the input of $\pi_x(P_1)^{1/2}\pi_x(T_0)$ to $F$.  The output
is supported in $KF$, giving a continuous finite matrix
\[
 W_F(x):\bigoplus_{v\in F}\mathbb C^{D_0}
       \longrightarrow
       \bigoplus_{v\in KF}\bigoplus_{s=1}^S\mathbb C^{R_s}.
\]
For an input vector $\xi$ satisfying $\pi_x(P_0)\xi=\xi$,
Equation~\eqref{eq:castle-finite-propagation-witness} implies
\begin{equation}
\label{eq:castle-window-injectivity}
 \|W_F(x)\xi\|^2\geq\tfrac12\|\xi\|^2.
\end{equation}
No coordinatewise assertion is substituted for this estimate: it holds
on the entire direct sum of the selected genuine projection ranges.

The output group $(s,v)$ of $W_F(x)$ vanishes unless
$T_vx\in B_s$.  Equation~\eqref{eq:allocated-output-rank-count} and
Lemma~\ref{lem:weighted-coordinate-compression} therefore give one
fixed linear map
\[
 C_F:\bigoplus_{v\in KF}\bigoplus_s\mathbb C^{R_s}
                    \longrightarrow\mathbb C^{q_F}
\]
injective on the range of $W_F(x)$ for every $x$.  The map $C_F$
depends on the finite window, but neither on $x$ nor on $z$.

\emph{The aligned bundle map.}
On $\mathcal T_F$, write $y_\nu=f_\nu(x)-z_\nu\in L$.
Choose a continuous function
$\chi:[-2,2]^t\to[0,1]$ with $\chi(0)=1$ and support contained in
$\{z:\|z\|_\infty<\delta\}$, and put
\[
 \theta_\nu(x,z)=w_\nu(x)\chi(z_\nu).
\]
Where $\theta_\nu>0$, write $\ell_\nu(x)=(s,v,j)$.  By
Equations~\eqref{eq:allocated-slot-agreement} and
\eqref{eq:castle-nearby-projection},
$\widetilde Q(F_{s,j}(T_vx))$ injects $\mathcal H_{y_\nu}$ into the $j$th
genuine projection range of the $s$th block of $P_0(T_vx)$.
Define
\[
 \Xi:\mathcal E_F\longrightarrow
              \mathcal T_F\times\bigoplus_{v\in F}\mathbb C^{D_0}
\]
by placing
\[
 \theta_\nu(x,z)\widetilde Q(F_{s,j}(T_vx))\eta_\nu
\]
in input group $(v,s,j)$ for $\eta_\nu\in \mathcal H_{y_\nu}$ when
$\theta_\nu>0$, and by placing zero when $\theta_\nu=0$.  Distinct
active slots have distinct groups by the allocation assumption.

This is continuous on the entire test locus.  On the positive-weight
region its formula and label are continuous.  At its complement, the
norm of each contribution is at most
$\theta_\nu\|\eta_\nu\|$, so the contribution tends to zero,
regardless of a change of label.  Equivalently, use the matrices
\[
 \theta_\nu(x,z)
 (\psi\widetilde Q)(F_{s,j}(T_vx))Q(y_\nu)
\]
on the ambient $\mathbb C^{2p}$, extended by zero in the inactive
groups.  There are only finitely many labels, and their uniformly
bounded entries vanish at every transition of label.  This proves
continuity without making any choice of a sparse-support projection.

The image of $\Xi_{(x,z)}$ lies in the identity subspace of
$\pi_x(P_0)$.  Hence the first map
\[
 \Phi_1(x,z)=C_FW_F(x)\Xi_{(x,z)}
           :\mathcal E_{F,(x,z)}\to\mathbb C^{q_F}
\]
has kernel exactly
\begin{equation}
\label{eq:allocated-first-map-kernel}
 \ker\Phi_1(x,z)=\bigoplus_{\nu:\theta_\nu(x,z)=0}\mathcal H_{y_\nu}.
\end{equation}
Indeed, $W_F(x)$ is injective on the image of $\Xi_{(x,z)}$ by
Equation~\eqref{eq:castle-window-injectivity}, $C_F$ is injective on
the range of $W_F(x)$, and the distinct input groups identify the
kernel of $\Xi$ with the displayed direct sum.

\emph{The sparse exceptional map.}
Use the fixed embeddings $\mathcal H_{y_\nu}\subset\mathbb C^{2p}$ to define
\[
 \Psi_{(x,z)}:\mathcal E_{F,(x,z)}
                  \longrightarrow\bigoplus_{\nu=1}^n\mathbb C^{2p},
 \qquad (\eta_\nu)_\nu\longmapsto
                  ((1-\theta_\nu(x,z))\eta_\nu)_\nu.
\]
If $1-\theta_\nu\neq0$, then either $z_\nu\neq0$, or both
$z_\nu=0$ and $w_\nu(x)<1$.  These disjoint possibilities occur in
at most $an$ and $\beta n$ slots, respectively, by the sparse
condition and Equation~\eqref{eq:allocated-slot-exception-count-on-locus}.
The original Equation~\eqref{eq:allocated-slot-exception-count} implies
the latter condition.  Thus the range
of $\Psi_{(x,z)}$ is supported on at most
$h_F=\lceil(a+\beta)n\rceil$ coordinate groups.  Apply
Lemma~\ref{lem:weighted-coordinate-compression}, now with every group
dimension equal to $2p$, to obtain a fixed map
\[
 D_F:\bigoplus_{\nu=1}^n\mathbb C^{2p}
                            \longrightarrow\mathbb C^{2ph_F}
\]
injective on every such support subspace.  Put $\Phi_2=D_F\Psi$.
The map $\Phi_1\oplus\Phi_2$ is fiberwise injective: a vector killed
by $\Phi_1$ has components only where $\theta_\nu=0$ by
Equation~\eqref{eq:allocated-first-map-kernel}; on these components
$\Psi$ is the identity, so injectivity of $D_F$ kills none of them.
Both maps are continuous, including at changes of sparse support.
They give the target rank in
Equation~\eqref{eq:allocated-bundle-target-rank}.  Locally the Gram
matrix of the resulting fiber map is positive definite; its inverse
square root makes the image a continuous rank-$pn$ subbundle.  The
test locus is compact, as required by the external theorem.

\emph{The output packing estimate.}
For fixed $x$, let
$\Lambda_s(x)=\{v\in KF:T_vx\in B_s\}$.  The sets
$F_sv$, with $s$ arbitrary and $v\in\Lambda_s(x)$, are pairwise
disjoint.  To verify this, suppose $av=bw$ with $a\in F_s$ and
$b\in F_t$.  Then the common point $T_{av}x=T_{bw}x$ belongs to both
tower levels $T_aB_s$ and $T_bB_t$.  Disjointness of levels forces
$(s,a)=(t,b)$, and cancellation then forces $v=w$.  Therefore
\[
 \sum_s|F_s|\,|\Lambda_s(x)|
 \leq\left|\left(\bigcup_sF_s\right)KF\right|.
\]
Multiplication by $\max_sR_s/|F_s|$ bounds the left-hand side of
Equation~\eqref{eq:allocated-output-rank-count}.  That left-hand side
is an integer, so its floor gives
Equation~\eqref{eq:castle-output-packing-bound}.
\end{proof}

\begin{remark}[Boundary padding is charged to the sparse parameter]
\label{rem:boundary-padding-sparse-charge}
The weaker exceptional-count assumption is useful when an allocation
has a variable number of real packets and is padded to a fixed number
of slots.  Give a padded slot weight zero and the fixed value
$b\in\partial[-1,1]^t$.  At a point of the torus test locus, this slot
must have $z_\nu\neq0$, because $b\notin L$.  Consequently it is
already charged to the sparse-parameter count and need not be charged
again to $\beta n$.  If padding is specified by a locally constant
macro-pattern wherever a macro weight $w$ is positive, the effective
slot weight $w$ on real slots and zero on padded slots is continuous:
the real-slot indicator is locally constant on that open region and
the product tends to zero where $w$ vanishes.  Slots in macros whose
weight is less than one may still be charged to the exceptional
count.  This observation concerns the analytic rank estimate; the
later cover-dimension count must separately subtract padded boundary
slots, which carry no real colors.
\end{remark}


\section{The individual-cover bound}
\label{sec:individual-cover}
\begin{theorem}\label{thm:cover-bound}
Under the free minimal URP hypotheses of Theorem~\ref{thm:actions-main}, every finite comparison bound $r$ for $A$ satisfies $d_\UU\leq2r$ for every finite open cover $\UU$.
\end{theorem}
\begin{proof}
It suffices to consider $r>0$: if zero is a comparison bound, every
positive $r$ is also a comparison bound, and we let $r\downarrow0$.
We argue by contradiction. We first fix the parameters and rank
estimates, then construct the simultaneous bundle embedding and
reconstruct a subordinate cover of smaller order.

\subsection{Parameters and quantitative estimates}\label{sec:actions-parameters}

Let $\mathcal U$ be a finite open cover and write
\[
 d=\lim_F\frac{\DD(\mathcal U^F)}{|F|}>2r,
 \qquad M=\max\{1,\DD(\mathcal U)\}.
\]
Here the limit is the Ornstein--Weiss cover rate, and $r\geq0$ is a
finite comparison bound.  Choose $p\geq1$ and then $0<a<1/16$ such that,
with $t=2p+1$ and $c=d/(t+1)$,
\[
 pc>r,
 \qquad G=(1-11a)pc-r>0.
\]
The following constants are fixed before either tiling is selected:
\begin{align*}
 \kappa=b&=\frac{a}{128},
 &\Delta&=\frac{ac}{32},
 &\tau&=\frac{G}{4(r+G)},\\
 \varepsilon&=\min\left\{\frac c4,
                  \frac{ac}{4096},\frac{\kappa c}{16}\right\},
 &c_-&=c-\varepsilon,
 &c_+&=c+\varepsilon,\\
 \eta&=\min\left\{\frac{\kappa}{4},\frac{ac}{4096M}\right\},
 &\delta_{\mathrm{out}}&=\min\left\{\frac a{512},
                                      \frac{ac}{4096M}\right\},
 &\delta_{\mathrm{prop}}&=\frac a{512}.
\end{align*}
In particular $c_->0$, $\kappa,b<a/64$, and $0<\tau\leq1/4$.
Fix also an Ornstein--Weiss invariance test ensuring
\begin{equation}\label{eq:urp-ow-reserve}
 \frac{\DD(\mathcal U^E)}{|E|}>d-\frac\Delta2
\end{equation}
for every finite nonempty set $E$ passing that test.  This test depends
only on $\mathcal U$ and the already fixed $\Delta$.

\paragraph{Primary shapes and their integer ranks.}
Choose an exact finite-shape primary tiling with shapes $F_s$ so invariant
and so large that
\[
 \left|\frac{\DD(\mathcal U^{F_s})}{|F_s|}-d\right|
       \leq\frac\varepsilon2,
 \qquad
 |F_s|\geq
 \max\left\{\frac4\varepsilon,\frac{4096}{apc}\right\}.
\]
Put $m_s=\DD(\mathcal U^{F_s})$,
\[
 g_s=\left\lfloor\frac{m_s+1}{t+1}\right\rfloor,
 \qquad R_s=\lfloor(2-10a)pg_s\rfloor.
\]
The preceding requirements give
\begin{equation}\label{eq:urp-primary-densities}
 c_-\leq\frac{g_s}{|F_s|}\leq c_+,
 \qquad \frac{m_s}{|F_s|}\leq d+\varepsilon.
\end{equation}
Moreover,
\begin{align*}
 \frac{R_s-pg_s}{|F_s|}
 &\geq(1-10a)p(c-\varepsilon)-\frac1{|F_s|}\\
 &\geq r+G+apc-\frac{apc}{4096}-\frac{apc}{4096}
 >r+G.
\end{align*}
Thus all floor errors are absorbed before the dynamical castle is chosen. Fix the colored primary nerves $\varphi_s:X\to S_s$ for $\UU^{F_s}$, of dimensions at most $m_s$, and their packet maps $F_{s,j}:X\to C_t$ now. Fix also a one-site nerve $\varphi_0:X\to S_0$ for $\UU$.

\paragraph{Macro shapes and padding.}
Having fixed the primary tiling, choose an exact macro tiling with shapes
$E_1,\ldots,E_L$ satisfying the allocation lemma's inner-boundary estimate
\[
 |\partial_H E_l|\leq\eta|E_l|,
 \qquad H=\bigcup_sF_sF_s^{-1},
\]
and the additional requirements
\[
 \DD(\mathcal U^{E_l})\leq(d+\varepsilon)|E_l|,
 \qquad |E_l|\geq\frac4{\kappa c}.
\]
Set
\[
 n_l=\lceil c_+|E_l|\rceil,
 \qquad B_l=(d+\varepsilon+M\eta)|E_l|+1.
\]
Fix fallback nerves $\varphi_l^{\mathrm{fb}}:X\to S_l^{\mathrm{fb}}$ for $\UU^{E_l}$ with dimensions at most $\DD(\UU^{E_l})$. For any pattern $\lambda$ furnished by the allocation below, let
$h_{l,\lambda}$ be the number of real packet slots in its disjoint primary
tiles.  The allocation lemma and \eqref{eq:urp-primary-densities} give
\[
 c_-(1-\eta)|E_l|\leq h_{l,\lambda}
      \leq c_+|E_l|\leq n_l.
\]
Consequently its number of padded slots satisfies
\begin{align*}
 \frac{n_l-h_{l,\lambda}}{n_l}
 &\leq\frac{2\varepsilon+c\eta+|E_l|^{-1}}{c_+}\\
 &\leq\frac{2\varepsilon}{c}+\eta+\frac{1}{c|E_l|}
 \leq\frac{\kappa}{8}+\frac{\kappa}{4}+\frac{\kappa}{4}
 <\kappa.
\end{align*}
Padded slot functions are fixed vectors in $\partial C_t$.  On the torus
test locus a padded slot therefore has nonzero sparse parameter.  Give each padded slot effective weight zero. On the torus test locus it is already included in the sparse-parameter exception budget $an$, by the weaker condition in Proposition~\ref{prop:allocated-bundle-embedding}. Padding is charged separately in the covering-order estimate below.

\paragraph{Closed exceptional sets are chosen last.}
All $n_l$ and $B_l$ are now fixed finite constants.  Define
\begin{equation}\label{eq:urp-closed-budget}
 \theta_l=\min\left\{
 \frac{b c_+(1-\delta_{\mathrm{out}})}{4L n_l},
 \frac{ac}{4096L B_l}\right\}>0.
\end{equation}
Only at this point invoke the allocation lemma, with primary sliced-tower
mass greater than $1-\tau$, and with compactly supported masks $w_l$ whose
exceptional closed sets $C_l\supset\{w_l<1\}$ satisfy
\[
 \sup_{\mu\in M_\Gamma(X)}\mu(C_l)<\frac{\theta_l}{2}.
\]
The uniform closed-set orbit estimate then supplies an invariance test,
depending on these already fixed $C_l$, such that
\begin{equation}\label{eq:urp-bad-count}
 \#\{g\in F:gx\in C_l\}\leq\theta_l|F|
 \quad(x\in X, 1\leq l\leq L)
\end{equation}
whenever $F$ passes that test.  Thus no exceptional-set tolerance is
chosen before the possibly large macro costs are known.

\paragraph{The strict tracial gap and the witness.}
Let $B_s$ be the sliced primary bases and arrange
\[
 \sum_s|F_s|\mu(B_s)>1-\tau
 \quad(\mu\in M_\Gamma(X)).
\]
For the positive elements encoding the primary packet bundles, their
target-minus-source dimension functions satisfy
\begin{align*}
 d_\mu(P_1)-d_\mu(P_0)
 &\geq\sum_s(R_s-pg_s)\mu(B_s)\\
 &>(r+G)(1-\tau)=r+\frac{3G}{4}>r.
\end{align*}
Cuntz comparison therefore gives one finite-propagation witness.  Its
finite support $K$ is fixed before the input window $F$ is chosen.

\paragraph{Input windows and source ranks.}
Choose $F$ after that witness.  In addition to
\eqref{eq:urp-bad-count} and \eqref{eq:urp-ow-reserve}, require that full
macro tiles contained in $F$ leave at most
$\delta_{\mathrm{out}}|F|$ uncovered sites and that
\begin{equation}\label{eq:urp-propagation-budget}
 \left|\left(\bigcup_sF_s\right)KF\right|
     \leq(1+\delta_{\mathrm{prop}})|F|.
\end{equation}
Let $q_l$ be the number of those macro tiles of type $l$, and set
$n_F=\sum_lq_ln_l$.  Their disjointness gives the useful explicit bounds
\begin{equation}\label{eq:urp-source-bounds}
 c_+(1-\delta_{\mathrm{out}})|F|
 \leq n_F
 \leq\left(c_++\frac{\kappa c}{4}\right)|F|
 \leq c_+\left(1+\frac{\kappa}{4}\right)|F|.
\end{equation}
The first term in \eqref{eq:urp-closed-budget}, together with
\eqref{eq:urp-bad-count}, ensures that the total number of slots in bad
macro tiles is at most
\[
 |F|\sum_l n_l\theta_l
 \leq\frac b4c_+(1-\delta_{\mathrm{out}})|F|
 \leq\frac b4 n_F\leq b n_F.
\]
The second term ensures that their additional join cost is at most
\begin{equation}\label{eq:urp-join-budget}
 |F|\sum_lB_l\theta_l\leq\frac{ac}{4096}|F|.
\end{equation}

\paragraph{A uniform strict target-rank margin.}
Disjointness of the active primary tiles and
\eqref{eq:urp-propagation-budget} bound the witness's output rank by
\[
 (2-10a)pc_+(1+\delta_{\mathrm{prop}})|F|.
\]
The fixed linear compression map and sparse exceptional map therefore
use total target rank at most
\[
 \left\lceil(2-10a)pc_+(1+\delta_{\mathrm{prop}})|F|\right\rceil
       +2p\lceil(a+b)n_F\rceil.
\]
Require also
\[
 |F|\geq
 \frac{2(1+2p)}{apc_+(1-\delta_{\mathrm{out}})}.
\]
After division by the source rank $pn_F$, the two ceiling errors are
then at most $a/2$.  The normalized target rank is consequently at most
\begin{align}\label{rank-control}
 &(2-10a)\frac{1+\delta_{\mathrm{prop}}}
                   {1-\delta_{\mathrm{out}}}
       +2(a+b)+\frac a2 \nonumber \\
 &\leq (2-10a)\left(1+\frac a{128}\right)
              +2a+\frac a{64}+\frac a2
 <2-7a<2.
\end{align}
The positive rank margin is independent of $F$ and of the
compression product exponent.

\paragraph{Covering-order saving with a fixed reserve.}
For a good macro tile, the uncompressed primary nerve factors and its
one-site hole factors cost at most
$(d+\varepsilon+M\eta)|E_l|$.  Each real boundary slot reduces this
carrier dimension by one.  A fallback macro nerve costs at most
$(d+\varepsilon)|E_l|$, so $B_l$ bounds the extra cost of the join at a
bad macro tile.  Cover uncovered input sites by one-site nerve factors.
If torus compression in $u$ copies supplies $(a/8)u n_F$ boundary slots,
discarding padding and all slots in bad macro tiles still leaves at least
\[
 (a/8-\kappa-b)u n_F
\]
real boundary slots in good macro tiles.  Hence the resulting cover
subordinate finite target complex has normalized dimension at most
\begin{align*}
 d+\varepsilon+M\eta+M\delta_{\mathrm{out}}
       +\frac{ac}{4096}
 -(a/8-\kappa-b)c_+(1-\delta_{\mathrm{out}}).
\end{align*}
The four positive error terms sum to at most $ac/1024$.  Furthermore
\[
 (a/8-\kappa-b)c_+(1-\delta_{\mathrm{out}})
 =\frac{7a}{64}c_+(1-\delta_{\mathrm{out}})
 >\frac{ac}{16}.
\]
Thus, for every permitted compression exponent $u$, the cover-rate upper
bound is strictly less than
\begin{equation}\label{eq:urp-final-reserve}
 d-\frac{63ac}{1024}<d-\Delta.
\end{equation}

\paragraph{The Ornstein--Weiss lower bound for every product exponent.}
For the left action used here, window translates are right translates.
Choose $c_1,\ldots,c_u$ such that the windows $Fc_j$ are pairwise disjoint,
and put $F_u=\bigcup_jFc_j$.  Every fixed left-invariance test passed by $F$ is also passed by $F_u$. In particular, for a finite test set $Q$ containing $e$,
\[
 |QF_u\setminus F_u|\leq u|QF\setminus F|.
\]
For each individual test element $k$,
\[
 |kF_u\mathbin{\triangle}F_u|
 \leq\sum_j|kFc_j\mathbin{\triangle}Fc_j|
 =u|kF\mathbin{\triangle}F|.
\]
The same observation applies to finite test sets.  Therefore
\eqref{eq:urp-ow-reserve} yields
\[
 \frac{\DD(\mathcal U^{F_u})}{|F_u|}>d-\frac\Delta2
 \qquad\hbox{for every }u\geq1,
\]
contradicting the subordinate-cover bound
\eqref{eq:urp-final-reserve}.  The final invariance and cardinality
requirements on $F$ are all imposed after the finite witness support is
known; increasing the invariance or cardinality of $F$ does not alter any
previously fixed constants, masks, packet maps, or ranks.

\subsection{Continuous slots and the bundle embedding}
\label{sec:compression-application}
We now construct the maps realizing the preceding budgets. Use the primary colored nerves $\varphi_s$, packet maps $F_{s,j}$, one-site nerve $\varphi_0$, and fallback nerves $\varphi_l^{\mathrm{fb}}$ fixed at the indicated stages above. Their dimensions are bounded by $m_s$, $m_0$, and $m(E_l)$, respectively.

Use the single allocation already chosen above from Proposition~\ref{prop:finite-shape-allocation}, with $\alpha=\tau$ and the fixed $\theta_l$ of \eqref{eq:urp-closed-budget}. Its bases $B_s$, compact sets $K_s$, pattern regions $O_{l,\lambda}$, and masks $w_l$ remain fixed. The positive functions and Fourier witness support $K$ are those already obtained from Proposition~\ref{prop:allocated-bundle-embedding} using the strict tracial gap. No allocation or witness is chosen anew.

For each macro type $l$ and pattern $\lambda$, order the pairs $(s,q)$ of its primary tiles, and then their $g_s$ packets. This orders its $h_{l,\lambda}$ real slots. Complete the list to $n_l$ slots using a fixed point $b_*\in\partial C_t$. Let
\[
 D_{l,\lambda}=\supp(w_l)\cap O_{l,\lambda}.
\]
For fixed $l$, these form finitely many disjoint compact sets. On $D_{l,\lambda}$ prescribe slot $\nu$ to have value $F_{s,j}(qx)$ when its label is the real triple $(s,q,j)$, and value $b_*$ when it is padded. Coordinatewise Tietze extension produces continuous maps
\[
 f_{l,\nu}:X\longrightarrow C_t,\qquad 1\leq\nu\leq n_l,
\]
with exactly these values on all $D_{l,\lambda}$. Thus on the entire positive-mask region the prescribed real packet values and padding values agree with the continuous slot functions.

Give slot $\nu$ the effective weight
\[
 w_{l,\nu}(x)=\begin{cases}
 w_l(x),&w_l(x)>0\text{ and its pattern label is real},\\
 0,&\text{otherwise}.
 \end{cases}
\]
It is continuous: its real-slot indicator is locally constant wherever $w_l>0$, and the product vanishes at the remaining transitions. A positive effective weight has a locally constant real label $(s,q,j)$ and $qx\in K_s$.

Let $\mathcal D_l(F)$ be the centers of macro tiles $E_lc\subset F$ in the fixed exact macro tiling. The retained tiles cover $F$ apart from a set $J_F$ of at most $\delta_{\mathrm{out}}|F|$ sites. This follows by asking that the inner collar of $F$ for $\bigcup_lE_lE_l^{-1}$ be that small: every tile through a core point is wholly contained in $F$. Such an invariance condition can be imposed after all preceding choices. Concatenate
\[
 f_{l,\nu}(cx),\quad
 l\leq L,\ c\in\mathcal D_l(F),\ \nu\leq n_l,
\]
in a fixed order, obtaining $f_F:X\to C_t^{n_F}$. The effective weights are $w_{l,\nu}(cx)$.

If a real label is $(s,q,j)$, its allocation label in Proposition~\ref{prop:allocated-bundle-embedding} is $(s,qc,j)$ and its input group is $(qc,s,j)$. Since $e\in F_s$, its center $q$ belongs to $E_l$, so $qc\in F$. The real primary tiles within a macro are disjoint; different retained macros are disjoint. Hence all positive effective weights have distinct input groups. Agreement with the original packet maps and compact containment follow from the allocation lemma.

On $\mathcal T(f_F,Z_{an_F})$, a padded slot in a macro with $w_l=1$ has value $b_*$, and therefore cannot have sparse parameter zero: $b_*\notin L$. Every zero-parameter slot with effective weight less than one is consequently in a macro with $w_l<1$. Their total slot count is at most $bn_F$, by \eqref{eq:urp-bad-count} and \eqref{eq:urp-source-bounds}. This verifies the weaker exception hypothesis \eqref{eq:allocated-slot-exception-count-on-locus} with $\beta=b$.

\paragraph{Application of \cite[Theorem 1.3]{OAI-top}.}
We now identify explicitly every datum in the cited theorem. Take
\[
 Y=X,\qquad n=n_F,\qquad f=f_F,
 \qquad Z=Z_{an_F}(t,n_F),
\]
with the integers $p,t=2p+1$ and the constant $0<a<1/16$ already fixed. The geometric data $L,\mathcal H$ are those specified in Section~\ref{sec:torus-compression}. The base and source bundle are exactly
\[
 \mathcal T_F=\mathcal T(f_F,Z_{an_F}),
 \qquad \mathcal E_{f_F}=d_{f_F}^*\mathcal H_{n_F}=\mathcal E_F,
\]
where $\mathcal E_F$ is the bundle in Proposition~\ref{prop:allocated-bundle-embedding}. In particular its fiber at $(x,z)\in\mathcal T_F$ is
\[
 (\mathcal E_{f_F})_{(x,z)}
 =\bigoplus_{\nu=1}^{n_F}
       \mathcal H_{(f_F)_\nu(x)-z_\nu},
 \qquad \rank\mathcal E_{f_F}=pn_F.
\]
The locus is compact, since $X$ and $Z_{an_F}(t,n_F)$ are compact and $L^{n_F}$ is closed. No dimension bound on $X$ is required.

For the target rank choose
\[
 q_F=\left\lceil(2-10a)pc_+(1+\delta_{\mathrm{prop}})|F|\right\rceil,
 \qquad k_F=q_F+2p\lceil(a+b)n_F\rceil.
\]
The output packing estimate gives \eqref{eq:allocated-output-rank-count} with this $q_F$. All other allocation hypotheses were verified above, including the torus-locus exception bound with $\beta=b$. Proposition~\ref{prop:allocated-bundle-embedding} therefore supplies the single map
\[
 \Phi_1\oplus\Phi_2:
 \mathcal E_{f_F}\longrightarrow\mathcal T_F\times\C^{k_F}.
\]
It is fiberwise injective and continuous on the whole locus, including changes of the support of $z$ and of the active allocation labels. The rank estimate of Section~\ref{sec:actions-parameters} (\eqref{rank-control}) gives $k_F/(pn_F)<2-7a<2$. Thus the nonnegative integer called $K$ in \cite[Theorem 1.3]{OAI-top} is $k_F<2pn_F$ here; it is unrelated to our propagation set $K\subset\Gamma$.

Every hypothesis of \cite[Theorem 1.3]{OAI-top}, restated as Theorem~\ref{input:compression}, is now satisfied. Its conclusion yields an integer $u\geq1$ and
\[
 \mathcal G:X^u\longrightarrow C_t^{un_F}.
\]
For every $\boldsymbol x=(x_1,\ldots,x_u)\in X^u$, at least $\lceil(a/8)un_F\rceil$ slots of $\mathcal G(\boldsymbol x)$ lie on $\partial C_t$. Writing $\mathcal G_{b,\nu}$ for copy $b$ and slot $\nu$, the preservation conclusion is
\[
 (f_F)_\nu(x_b)\in\partial C_t
 \quad\Longrightarrow\quad
 \mathcal G_{b,\nu}(\boldsymbol x)=(f_F)_\nu(x_b).
\]
This is equality of the entire $t$-dimensional vector. The theorem permits $u$ arbitrarily large; our proof uses only the existence of one such exponent.

\subsection{Reconstruction of the whole orbit cover}
Choose $c_1,\ldots,c_u\in\Gamma$ such that $Fc_b$ are pairwise disjoint, and put $F_u=\bigsqcup_bFc_b$. Such choices exist because $\Gamma$ is infinite: at each stage only finitely many new translates are forbidden. Pull $\mathcal G$ back along
\[
 \iota_u(x)=(c_1x,\ldots,c_ux).
\]
Write its replacement slots as $v_{b,l,c,\nu}(x)$. In a macro with positive weight, a real slot $\nu$ labeled $(s,q,j)$ has original value $F_{s,j}(qcc_bx)$. The preservation conclusion just stated therefore gives
\[
 F_{s,j}(qcc_bx)\in\partial C_t
 \quad\Longrightarrow\quad
 v_{b,l,c,\nu}(x)=F_{s,j}(qcc_bx).
\]
This supplies exactly the boundary-preservation hypothesis of Lemma~\ref{lem:packets} for each primary nerve factor on its open pattern region. The boundary-slot lower bound also holds at every $\iota_u(x)$, and will supply the carrier-dimension saving below.

For each macro pattern $\lambda$ define the finite polyhedron
\[
 A_{l,\lambda}=
 \prod_{(s,q)\in\mathcal P_{l,\lambda}}S_s
 \ \times\prod_{q\in J_{l,\lambda}}S_0,
 \qquad A_l=\bigsqcup_{\lambda\in\Lambda_l}A_{l,\lambda}.
\]
The primary coordinates before reconstruction are $\varphi_s(qcc_bx)$; the hole coordinates are $\varphi_0(qcc_bx)$. Reconstruct each primary factor using its assigned replacement packets. The resulting point in $A_{l,\lambda}$ is continuous on its positive-mask pattern region, and every positive vertex coordinate still implies membership in the corresponding original orbit-cover member.

Join this pattern point to the unchanged fallback point $\varphi_l^{\mathrm{fb}}(cc_bx)$ in
\[
 P_l=A_l*S_l^{\mathrm{fb}},
\]
giving the $A_l$ branch weight $w_l(cc_bx)$ and the fallback branch weight $1-w_l(cc_bx)$. This defines a continuous map into $P_l$ even where the pattern changes: the bounded coordinates of the entire $A_l$ branch are multiplied by a weight tending to zero. At weight one the fallback branch disappears. Include the unchanged one-site factors $\varphi_0(qc_bx)$ for $q\in J_F$. Taking the product over all $b,l,c$ and these outside sites defines a continuous map
\[
 \Phi:X\longrightarrow P_u,
\]
where $P_u$ is a finite polyhedron.

We verify subordination on the whole inverse image of each target
neighborhood. At any point of $\Phi(X)$, choose in each macro factor a branch of positive weight. If it is the pattern branch, its component $\lambda$ is locally fixed; choose a positive vertex in each primary and hole factor. The primary cover members describe all sites of $F_sq$, and the hole factors describe the complementary sites, together covering $E_l$ exactly. Preservation of vertex support ensures the same memberships throughout the inverse image of a neighborhood on which those vertices remain positive. If the fallback branch is chosen, a positive vertex in its unchanged nerve specifies a member of $\UU^{E_l}$ on that whole inverse image. Make the analogous choices for the outside one-site factors. The finitely many assigned members combine to one member of $\UU^{F_u}$. Intersecting the positive-coordinate neighborhoods proves the hypothesis of Lemma~\ref{lem:refinement}.

For completeness, the dimension estimate is an estimate of a fixed finite target subcomplex. Equip $P_u$ with the cell structure formed from products and joins of simplex faces in its factors. For a macro of type $l$ with weight one, its carrier dimension is at most
\[
 \sum_{(s,q)\in\mathcal P_{l,\lambda}}m_s
       +m_0|J_{l,\lambda}|-s_{b,l,c}(x),
\]
where $s_{b,l,c}(x)$ counts its real boundary replacement packets. The expression before the saving is at most $(d+\varepsilon+M\eta)|E_l|$. For a macro with weight less than one, the product-join carrier bound is at most the baseline plus $m(E_l)+1$, which is at most the baseline plus $B_l$. This bound also covers weight zero, when only the fallback branch remains.

At most $\kappa un_F$ boundary slots can be padded slots in good macros, and at most $bun_F$ can be in bad macros. The remaining real boundary slots in good macros total at least
\[
 (a/8-\kappa-b)un_F.
\]
The estimates \eqref{eq:urp-join-budget} and \eqref{eq:urp-source-bounds}, together with the outside-site cost $M\delta_{\mathrm{out}}u|F|$, therefore bound every image carrier by
\begin{align*}
 u|F|\bigl(d+\varepsilon+M\eta+M\delta_{\mathrm{out}}
       +ac/4096\bigr)
 -(a/8-\kappa-b)un_F
 <u|F|(d-\Delta),
\end{align*}
as computed in \eqref{eq:urp-final-reserve}. There are finitely many cells, and the faces of a cell never have larger dimension. Thus $\Phi(X)$ lies in the union of all cells of dimension at most $\lfloor u|F|(d-\Delta)\rfloor$, a finite subcomplex of that dimension. Triangulate this subcomplex without raising its dimension. Lemma~\ref{lem:refinement} yields
\[
 \frac{\DD(\UU^{F_u})}{|F_u|}\leq d-\Delta.
\]
On the other hand, $F_u$ passes the Ornstein--Weiss lower test because $F$ does, giving \eqref{eq:urp-ow-reserve} for $F_u$. This is a contradiction.
\end{proof}

\subsection{The lower bound and equality}
\begin{proof}[Proof of Theorem~\ref{thm:actions-main}]
Theorem~\ref{thm:cover-bound} proves $d_\UU\leq2r$ for every finite comparison bound $r$ and every finite open cover $\UU$. Taking the supremum over covers gives $\mdim(X,\Gamma)\leq2r$. If comparison bounds exist, taking their infimum proves the first assertion of Theorem~\ref{thm:actions-main}; if none exist, its right-hand side is infinite and the assertion is automatic. In particular, infinite mean dimension excludes every finite comparison bound. Under COS, the upper bound in  \cite[Theorem 4.8]{Niu} completes the equality in the extended nonnegative reals.
\end{proof}

\begin{remark}[The hypotheses used]
Exact group tilings and URP provide allocation, Cuntz comparison
supplies the finite-propagation witness, and
Theorem~\ref{input:compression} supplies boundary-preserving compression.
COS enters through the established upper bound. The theorem retains
URP as a hypothesis for its lower bound and URP together with COS
for equality.
\end{remark}

\begin{remark}[Order of choices]
First fix $p,a$ and the numerical margins. Then choose the primary
tiling and packet maps, followed by the macro tiling and slot counts.
Next choose the small-error URP allocation and masks, and fix one Fourier
witness. Finally choose the window $F$ and then the exponent $u$.
The closed-set budgets are chosen after all macro costs are known, and the final window after propagation is known. Neither choice changes earlier data. Right-translated product windows preserve all fixed left-invariance tests, so there is no restriction on the exponent furnished by compression.
\end{remark}

\end{document}